\documentclass[journal]{IEEEtran}

\usepackage{amsthm, amsmath, amsfonts, amssymb, graphicx, enumerate, subfig}

\allowdisplaybreaks
\newcommand{\R}{\mathbb{R}}
\newcommand{\N}{\mathbb{N}}

\newcommand{\norm}[1]{\left|\left| #1 \right|\right|}
\newcommand{\innerprod}[2]{\left\langle #1,#2 \right\rangle}

\newtheorem{theorem}{Theorem}
\newtheorem{lemma}[theorem]{Lemma}
\newtheorem{assum}[theorem]{Assumption}
\newtheorem*{remark}{Remark}

\title{Folding Bilateral Backstepping Output-Feedback Control Design For an Unstable Parabolic PDE}
\author{
	Stephen Chen$^{1}$~\IEEEmembership{Student Member,~IEEE,}, Rafael Vazquez$^{2}$~\IEEEmembership{Member,~IEEE,}, and Miroslav Krstic$^{1}$~\IEEEmembership{Fellow,~IEEE,}
	\thanks{$^{1}$Stephen Chen and Miroslav Krstic are with the Department of Mechanical and Aerospace Engineering, University of California, San Diego, La Jolla, CA 92093-0411 USA (e-mail: stc007@ucsd.edu and krstic@ucsd.edu, respectively).}%
	\thanks{$^{2}$Rafael Vazquez is with the Departmento de Ingenier\'ia Aeroespacial, Universidad de Sevilla, 41092 Sevilla, Spain (e-mail: rvazquez1@us.es).}%
}

\begin{document}
\maketitle


\begin{abstract}
	We present a novel methodology for designing output-feedback backstepping boundary controllers for an unstable 1-D diffusion-reaction partial differential equation with spatially-varying reaction. Using ``folding transforms'' the parabolic PDE into a $2 \times 2$ coupled parabolic PDE system with coupling via folding boundary conditions. The folding approach is novel in the sense that the design of bilateral controllers are generalized to center around arbitrary points, which become additional design parameters that can be separately chosen for the state-feedback controller and the state observer. The design can be selectively biased to achieve different performance indicies (e.g. energy, boundedness, etc). A first backstepping transformation is designed to map the unstable system into a strict-feedback intermediate target system. A second backstepping transformation is designed to stabilize the intermediate target system. The invertibility of the two transformations guarantees that the derived state-feedback controllers exponentially stabilize the trivial solution of the parabolic PDE system in the $L^2$ norm sense. A complementary state observer is likewise designed for the dual problem, where two collocated measurements are considered at an arbitrary point in the interior of the domain. The observer generates state estimates which converge to the true state exponentially fast in the $L^2$ sense. Finally, the output feedback control law is formulated by composing the state-feedback controller with the state estimates from the observer, and the resulting dynamic feedback is shown to stabilize the trivial solution of the interconnected system in the $L^2$ norm sense. Some analysis on how the selection of these points affect the responses of the controller and observer are discussed, with simulations illustrating various choices of folding points and their effect on the stabilization.
\end{abstract}

\section{Introduction}
Parabolic partial differential equations (PDEs) describe numerous physical processes, which include but are not limited to heat transfer, chemical reaction-diffusion processes, tumor angiogenesis \cite{doi:10.1146/annurev.bioeng.8.061505.095807}, predator-prey Lotka-Volterra population models \cite{Hastings1978}, opinion dynamics (of the Fischer-Kolmogorov-Petrovsky-Piskunov type equation \cite{opiniondynamics}), free-electron plasma diffusion, and flows through porous media \cite{vazquez2007porous}.

Previous results in boundary control for 1-D PDEs has been largely focused on unilateral boundary controllers, i.e. controllers acting on a single boundary. Results have been generated for a wide variety of parabolic PDE systems and objectives, beginning with the classical scalar 1-D PDE with homogeneous media results \cite{krsticbacksteppingpde}. Other extensions to the parabolic PDE boundary control case introduce nonhomgeneous media (such as \cite{smyshlyaev2005spacetime}), parallel interconnected parabolic PDE systems \cite{vazquez2016coupledpara}, series interconnected parabolic PDE systems \cite{tsubakino2013boundary}, and output feedback extensions for coupled parabolic PDE \cite{orlov2017outputfdback}. Some work that is tangentially related is that of \cite{wang2013backstepping},\cite{woittennek2014backstepping}, which investigates a problem of using an in-domain actuation to control a parabolic PDE.

The notion of bilateral boundary control is partially motivated by boundary control of balls in $\R^n$ \cite{vazquez2017boundary}, in which the controls actuate on the surface of the $n$-dimensional ball. The analogous case (in 1-D) is a controller actuating on the boundary of the $1$-dimensional ball, i.e., the endpoints of an interval. Bilateral control has been studied in some contexts for both hyperbolic and parabolic PDE systems. \cite{auriol2018two} studies bilateral controllers achieving minimum-time convergence in coupled first-order hyperbolic systems via a Fredholm transformation technique, while \cite{vazquez2016bilateral} additionally studies bilateral control for diffusion-reaction equations, albeit with the limitatin of a symmetric Volterra transformation. \cite{nikos2018bilateralhj} studies a nonlinear viscous Hamilton-Jacobi PDE, which likewise uses the symmetric Volterra transformation from \cite{vazquez2016bilateral}.

Boundary observer design is of equal (and perhaps arguably more) importance when compared with the boundary controller design. Many results have been generated as a dual problem to the boundary controller case. In \cite{SMYSHLYAEV2005613}, a boundary observer design for parabolic PDEs is formulated, with measurements taken at a boundary (in both collocated and anticollocated cases). \cite{baccoli2015anticollocated} studies a coupled parabolic PDE system with identical diffusion coefficients. \cite{camacho2017coupled} recovers a result for coupled parabolic PDEs with varying diffusion coefficients.

The main contribution of the paper are results for bilateral control of diffusion-reaction equations with spatially-varying reaction via the method of ``folding,'' i.e. using an arbitrarily defined domain separation and transformation to design the boundary controllers. The idea of folding has been touched upon in the hyperbolic context \cite{de2016boundary}, where the authors have explored a linearized Rijke tube model. The folding technique admits a design parameter (called the folding point) whose choice influences the control effort exerted by the boundary controllers. Additionally, a state-estimator is designed to complement the state-feedback controller. The state-estimator is an interesting new development in which collocated measurements are taken from any arbitrary point in the interior of the PDE, and the folding approach applied. Finally, the output feedback is formulated by combining both the state-feedback and state-estimation.

The state observer design is an interesting development, as it generates a result where measurements are taken at a single measure zero point in the interior. It is of physical importance, as measurements at the boundary are not necessarily guaranteed for a given realization. \cite{tsubakino2015weighted} has also investigated observer designs where measurements are not given at a boundary, rather, as a weighted average (the state appearing underneath a bounded integral operator). A related result is \cite{FREUDENTHALER20176780}, in which the authors consider the combination of boundary measurements with a single interior measurement to achieve estimation convergence for semilinear parabolic problems.

The primary technical difficulty in the paper is compensating the folding-type boundary conditions, which arises due to the regularity property of the solutions. In hyperbolic PDE, this constitutes an imposition of continuity -- a first-order compatibility condition. However, in parabolic PDE, one must treat second-order compatibility conditions existing at the same point, which will require additional correctional designs to compensate.

The paper is organized as follows: the notations and model are introduced in Section \ref{sec:prelim}. The output feedback controller consisting of the state-feedback and state-estimator designs is developed in Section \ref{sec:controller}. The gain kernel well posedness is studied in Section \ref{sec:wellposed}. Some simulations for various folding scenarios are given and analyzed in Section \ref{sec:simulation}. Finally, the paper is concluded in Section \ref{sec:conclusions}.

\section{Preliminaries}
\label{sec:prelim}
\subsection{Notation}
The partial operator is notated using the del-notation, i.e.
\begin{align*}
	\partial_x f &:= \frac{\partial f}{\partial x}
\end{align*}

$L^2(I_o)$ is defined as the the $L^2$ space on the interval $I_o$, equipped with the norm
\begin{align*}
	\norm{f}_{L^2(I_o)} &= \left(\int_{I_o} f^2 \,\,d\mathbf{\mu} \right)^{\frac{1}{2}}
\end{align*}

We also consider the standard inner product (that induces the standard norm) for $L^2$:
\begin{align*}
	\innerprod{f}{g}_{L^2(I_o)} &= \int_{I_o} f \cdot g \,\,d\mathbf{\mu}
\end{align*}

For compact notation, we will let $L^2(I_o)$ be represented merely as $L^2$, where the interval is implied by the function.
The norm notation ($\norm{\cdot}$) is used to notate the function norm over the vector 2-norm (notated with $|\cdot|_2$). If the norm is taken over a matrix function, then the induced 2-norm is implied, i.e. for vector-valued function $f$ and matrix-valued function $F$:
\begin{align*}
	\norm{f}_{L^2} := \left(\int_{I_o} \norm{f}_2^2 d\mathbf{\mu}\right)^{\frac{1}{2}}
	\norm{F}_{L^2} := \left(\int_{I_o} \norm{F}_{2,i}^2 d\mathbf{\mu}\right)^{\frac{1}{2}}
\end{align*}

Furthermore, if $f$ is a function of the space-time tuple $(x,t)$, the norm is assumed to be the norm in space ($x$) unless otherwise stated. The written $x$-dependence is dropped, i.e.
\begin{align*}
	\norm{f(x,t)}_{L^2} = \norm{f(t)}_{L^2} := \left(\int_{I_o} |f(t)|_2^2 \,\, d\mathbf{\mu} \right)^{\frac{1}{2}}
\end{align*}

We will introduce the notion of stability in the sense of a norm. Rigorously, this refers to the norm in which stability is derived. Per example, stability in the sense of $L^2$ refers to a stability estimate using $L^2$ norms:
\begin{align*}
	\norm{f(t)}_{L^2} \leq M \norm{f(t_0)}_{L^2}
\end{align*}

Elements of a matrix $A$ are denoted with lowercase $a_{ij}$, with the subscripts defining the $i$-th row and $j$-th column.

\subsection{Model and problem formulation}
We consider the following reaction-diffusion PDE for $u$ on the domain $[0,\infty) \times (-1,1)$:
\begin{align}
	\partial_t \bar{u}(y,t) &= \varepsilon \partial_{y}^2 \bar{u}(y,t) + \nu(y) \partial_y \bar{u}(y,t) + \bar{\lambda}(y) \bar{u}(y,t) \label{eq:model_orgadv_1} \\
	\bar{u}(-1,t) &= \bar{\mathcal{U}}_1(t) \label{eq:model_orgadv_2}\\
 	\bar{u}(1,t) &= \bar{\mathcal{U}}_2(t) \label{eq:model_orgadv_3}
\end{align}
It is assumed that $\varepsilon > 0$ for well-posedness, and $\nu, \bar{\lambda} \in C^1((-1,1))$. The controllers operate at $x = 1$ and $x = -1$, and are denoted $\bar{\mathcal{U}}_1(t),\bar{\mathcal{U}}_2(t)$, respectively. We define the following transformation:
\begin{align}
	u(y,t) &= \exp \left( \int_{-1}^y \frac{\nu(z)}{2 \varepsilon} dz \right) \bar{u}(y,t) \label{eq:model_adv_transform}
\end{align}
and with the appropriate parameter definitions, we find the equivalent system
\begin{align}
	\partial_t u(y,t) &= \varepsilon \partial_{y}^2 u(y,t) + \lambda(y) u(y,t) \label{eq:model_org_1} \\
	u(-1,t) &= \bar{\mathcal{U}}_1(t) =: \mathcal{U}_1(t) \label{eq:model_org_2}\\
 	u(1,t) &= \exp \left( \int_{-1}^{1} \frac{\nu(z)}{2 \varepsilon} dz \right) \bar{\mathcal{U}}_2(t) =: \mathcal{U}_2(t) \label{eq:model_org_3}
\end{align}
The transformation \eqref{eq:model_adv_transform} removes the advection/convection term in \eqref{eq:model_orgadv_1}. The attenuation/amplification of control effort in the controllers matches intuition -- the controller upstream of the ``average'' convection requires less control effort, while the controller downstream requires more control effort (average, as the sign of $\nu$ can vary across the domain). In this paper, we will assume that $\nu = 0$, but in general, the methodology can compensate convection phenomena.

\section{Output-feedback control design}
\label{sec:controller}
The output feedback is designed via solving two subproblems: the state-feedback design, and state-estimator design. The output-feedback result is then recovered by replacing the state-feedback control law with the state-estimate, and the resulting stability of the interconnected systems is proven.

\subsection{Model transformation for control via folding}
The folding approach entails selecting a point $y_0 \in (-1,1)$ in which the scalar parabolic PDE system $u$ is ``folded'' into a $2 \times 2$ coupled parabolic system. A special case $y_0 = 0$ (dividing into a symmetric problem) recovers the result of \cite{vazquez2016bilateral}. We define the the folding spatial transformations as
\begin{align}
	x &= (y_0 - y)/(1 + y_0) & y \in (-1,y_0) \\
	x &= (y - y_0)/(1 - y_0) & y \in (y_0,1)
\end{align}
admits the following states:
\begin{align}
	U(x,t) &:= \begin{pmatrix} u_1(x,t) \\ u_2(x,t) \end{pmatrix} = \begin{pmatrix} u(y_0 - (1+y_0)x,t) \\ u(y_0 + (1-y_0)x,t) \end{pmatrix} \label{eq:fold_tfm}
\end{align}
whose dynamics are governed by the following system:
\begin{align}
	\partial_t U(x,t) &= E \partial_{x}^2 U(x,t) + \Lambda(x) U(x,t) \label{eq:model_fold_1} \\
	\alpha U_x(0,t) &= -\beta U(0,t) \label{eq:model_fold_2}\\
	U(1,t) &= \mathcal{U}(t) \label{eq:model_fold_3}
\end{align}
with the parameters given by :
\begin{align}
	E &:= \text{diag}(\varepsilon_1,\varepsilon_2) \nonumber\\ &:= \text{diag}\left(\frac{\varepsilon}{(1+y_0)^2},\frac{\varepsilon}{(1-y_0)^2}\right) \\
	\Lambda(x) &:= \text{diag}(\lambda_1(x),\lambda_2(x)) \nonumber\\ &:= \text{diag}(\lambda(y_0 - (1+y_0)x), \lambda(y_0 + (1-y_0)x)) \\
	\alpha &:= \begin{pmatrix} 1 & a \\ 0 & 0 \end{pmatrix} \\
	\beta &:= \begin{pmatrix} 0 & 0 \\ 1 & -1 \end{pmatrix} \\
	a &:= (1+y_0)/(1-y_0) \label{eq:a_def}
\end{align}
The boundary conditions at $x = 0$ are curious. While they may initially appear to be encapsulated as Robin boundary conditions in \eqref{eq:model_fold_2}, they are actually compatibility conditions that arise from imposing continuity in the solution at the folding point. Analogous conditions have been considered in some previous parabolic backstepping work in \cite{tsubakino2013boundary}, albeit in a differing context.

\begin{assum}
	The folding point $y_0$ is constricted to the half domain $(-1,0]$ without loss of generality. The case $y_0 \in [0,1)$ can be recovered by using a change in spatial variables $\hat{y} = -y$ and performing the same folding technique. By choosing $y_0$ in this manner, we impose an ordering $\varepsilon_1 > \varepsilon_2$.
	\label{assum:order}
\end{assum}

\begin{figure}[tb]
	  \centering
		\includegraphics[width=\linewidth]{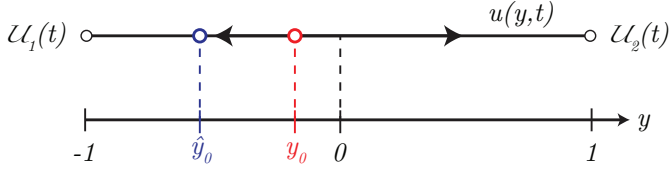}
		\caption{System schematic of diffusion-reaction equation with two boundary inputs. The control folding point $y_0$ and the measurement location $\hat{y}_0$ can be arbitrarily chosen on the interior, independent of one another.}
		\label{fig:system_schem}
\end{figure}

\subsection{State-feedback design}
The backstepping state-feedback control design is accomplished with two consecutive backstepping transformations. The first transformation is a $2 \times 2$ Volterra integral transformation of the second kind:
\begin{align}
	W(x,t) &= U(x,t) - \int_0^x K(x,y) U(y,t) dy \label{eq:tfm1}
\end{align}
where $K(x,y) \in C^2(\mathcal{T})$ is a $2 \times 2$ matrix of kernel elements $(k_{ij})$, with $\mathcal{T} := \{ (x,y) \in \R^2 | 0 \leq y \leq x \leq 1 \}$, and $W(x,t) := \begin{pmatrix} w_1(x,t) & w_2(x,t)\end{pmatrix}^T$. The inverse transformation is analogous:
\begin{align}
	U(x,t) &= W(x,t) - \int_0^x \bar{K}(x,y) W(y,t) dy \label{eq:tfm1inv}
\end{align}
The corresponding target system for \eqref{eq:tfm1} is chosen to be
\begin{align}
	\partial_t W(x,t) &= E \partial_x^2 W(x,t) - C W(x,t) + G[K](x) W(x,t)  \label{eq:targ_int_1} \\
	\alpha \partial_x W(0,t) &= -\beta W(0,t) \label{eq:targ_int_2}\\
	W(1,t) &= \mathcal{V}(t) \label{eq:targ_int_3}
\end{align}
where $\mathcal{V}(t) = \begin{pmatrix} 0 & \nu_2(t) \end{pmatrix}^T$ is an auxiliary control which is designed later in the paper. The controller $\mathcal{U}(t)$ can be expressed as an operator of $\mathcal{V}(t)$ by evaluating \eqref{eq:tfm1} for $x = 1$:
\begin{align}
	\mathcal{U}(t) := \mathcal{V}(t) + \int_0^1 K(1,y) U(y,t) dy \label{eq:control_aux}
\end{align}
The matrix $C$ can be arbitrarily chosen such that $C \succ 0$, but for simplicity of analysis, we select a diagonal matrix $C = \text{diag}(c_1,c_2)$ with $c_1,c_2 > 0$. The matrix-valued operator $G[\cdot](x)$ acting on $K$ is given by
\begin{align}
	G[K](x) = \begin{pmatrix} 0 & 0 \\ (\varepsilon_2 - \varepsilon_1) \partial_y k_{21}(x,x) & 0 \end{pmatrix} =: \begin{pmatrix} 0 & 0 \\ g[k_{21}](x) & 0 \end{pmatrix} \label{eq:G_op_def}
\end{align}
Imposing  the conditions \eqref{eq:model_fold_1}, \eqref{eq:model_fold_2}, \eqref{eq:tfm1}, \eqref{eq:targ_int_1}-\eqref{eq:targ_int_3} admits the following companion gain kernel PDE system for $K(x,y)$:
\begin{align}
	E \partial_x^2 K(x,y) - \partial_y^2 K(x,y) E &= K(x,y) \Lambda(y) + C K(x,y)\nonumber\\&\quad - G[K](x) K(x,y) \label{eq:K_mat_PDE}\\
	\partial_y K(x,x) E + E \partial_x K(x,x) &= - E \frac{d}{dx} K(x,x) -\Lambda(x) \nonumber\\&\quad- C + G[K](x) \label{eq:K_xx_bc}\\
	E K(x,x) - K(x,x) E &= 0 \label{eq:K_xtra} \\
	K(x,0) E \partial_x U(0) &= \partial_y K(x,0) E U(0) \label{eq:K_fold}
\end{align}
It is clear to see that by imposing \eqref{eq:K_xtra}, the definition for $G[K](x)$ can be recovered from \eqref{eq:K_xx_bc}
Upon first inspection, the resulting kernel PDE is very similar to those found in \cite{vazquez2016coupledpara},\cite{deutscher2018backstepping}. However, one may see that \eqref{eq:K_fold} is different, and in fact quite new in backstepping designs. \eqref{eq:K_fold} arises due to the folding boundary condition \eqref{eq:model_fold_2}. Surprisingly enough, if one analyzes \eqref{eq:K_fold} componentwise and employs \eqref{eq:model_fold_2}, ``anti-folding'' conditions on $K$ can be recovered, which preserve continuity in the spatial derivative of the state (as opposed to folding conditions preserving continuity in the state). The folding conditions that arise from \eqref{eq:K_fold} are:
\begin{align}
	\varepsilon_1 k_{11}(x,0) - a \varepsilon_2 k_{12}(x,0) &= 0 \\
	\varepsilon_1 \partial_y k_{11}(x,0) + \varepsilon_2 \partial_y k_{12}(x,0) &= 0 \\
	\varepsilon_1 k_{21}(x,0) - a \varepsilon_2 k_{22}(x,0) &= 0 \\
	\varepsilon_1 \partial_y k_{21}(x,0) + \varepsilon_2 \partial_y k_{22}(x,0) &= 0
\end{align}
or, more compactly written,
\begin{align}
	\tilde{\alpha} K(x,0) &= \tilde{\beta} \partial_y K(x,0) \label{eq:K_fold_compact}
\end{align}
where
\begin{align}
	\tilde{\alpha} &:= \begin{pmatrix} 1 & -a \\ 0 & 0 \end{pmatrix}, \quad \tilde{\beta} := \begin{pmatrix} 0 & 0 \\ 1 & 1 \end{pmatrix}
\end{align}
The kernel equations for the inverse kernels $\bar{K}$ are similar to those of $K$, and are derived in an analogous manner:
\begin{align}
	E \partial_x^2 \bar{K}(x,y) - \partial_y^2 \bar{K}(x,y) E &= -\bar{K}(x,y) (C-G[\bar{K}](y))\nonumber\\&\quad - \Lambda(x) \bar{K}(x,y) \label{eq:K_bar_mat_PDE}\\
	\partial_y \bar{K}(x,x) E + E \partial_x \bar{K}(x,x) &= - E \frac{d}{dx} \bar{K}(x,x) +\Lambda(x) \nonumber\\&\quad + C - G[\bar{K}](x) \label{eq:K_bar_xx_bc}\\
	E \bar{K}(x,x) - \bar{K}(x,x) E &= 0 \label{eq:K_bar_xtra} \\
	\bar{K}(x,0) E \partial_x W(0) &= \partial_y \bar{K}(x,0) E W(0) \label{eq:K_bar_fold}
\end{align}

The second transformation is designed to admit an expression for the auxiliary controller $\mathcal{V}(t) = \begin{pmatrix} 0 & \nu_2(t) \end{pmatrix}^T$. The goal of $\nu_2(t)$ is to remove the potentially destabilizing effect of the coupling term $G[K](x)$. The second set of transformations is:
\begin{align}
	\omega_1(x,t) &= w_1(x,t) \\
	\omega_2(x,t) &= w_2(x,t) - \int_0^x \begin{pmatrix} q(x,y) & p(x-y) \end{pmatrix} W(y,t) dy \nonumber\\
	&\qquad\qquad\,\,\,\,- \int_x^1 \begin{pmatrix} r(x,y) & 0 \end{pmatrix} W(y,t) dy \label{eq:tfm2}
\end{align}
Let $\Omega(x,t) := \begin{pmatrix} \omega_1(x,t) & \omega_2(x,t)\end{pmatrix}^T$. The inverse transformations are given by
\begin{align}
	w_1(x,t) &= \omega_1(x,t) \\
	w_2(x,t) &= \omega_2(x,t) - \int_0^x \begin{pmatrix} \bar{q}(x,y) & \bar{p}(x-y) \end{pmatrix} \Omega(y,t) dy \nonumber\\
	&\qquad\qquad\,\,\,\,- \int_x^1 \begin{pmatrix} \bar{r}(x,y) & 0 \end{pmatrix} \Omega(y,t) dy \label{eq:tfm2inv}
\end{align}
We impose the following target system dynamics:
\begin{align}
	\partial_t \Omega(x,t) &= E \partial_x^2 \Omega(x,t) - C \Omega(x,t) \label{eq:targ2_int_1} \\
	\alpha \partial_x \Omega(0,t) &= -\beta \Omega(0,t) \label{eq:targ2_int_2}\\
	\Omega(1,t) &= 0 \label{eq:targ2_int_3}
\end{align}
Noting that $G[K](x)$ is parametrized by the difference in diffusion coefficients $\varepsilon_1 - \varepsilon_2$, one can interpret \eqref{eq:tfm2} to be the \emph{correction} factor to the first transformation in presence of selecting a non-trivial folding point. Indeed, when the folding point is chosen to be the midpoint, $G[K](x) \equiv 0$ (and therefore \eqref{eq:tfm2} becomes an identity transformation). This necessity for correction factors is to compensate for the behavior unique to bilateral control design in parabolic PDE, and is not observed in the results featuring bilateral control design of hyperbolic PDE systems \cite{auriol2018two}.

The transformation \eqref{eq:tfm2} features two major components -- a Volterra integral operator in $w_2$ characterized by kernel $p$, a Volterra integral operator in $w_1$ characterized by kernel $q$, and an \emph{forwarding} type of transformation in $w_1$ characterized by kernel $r$. The kernels $p,q$ are defined on the domain $\mathcal{T}$, while $r$ is defined on the domain $\mathcal{T}_\textrm{u} := \{ (x,y) \in \R^2 | 0 \leq x \leq y \leq 1 \}$.

The transformation \eqref{eq:tfm2} with the conditions \eqref{eq:targ_int_1}-\eqref{eq:targ_int_3},\eqref{eq:targ2_int_1}-\eqref{eq:targ2_int_3} imposed will admit the following definition for $p$ and kernel PDE for $q$:
\begin{align}
	p(x) &= a^{-1} q(x,0) \label{eq:p_orig} \\
	\varepsilon_2 \partial_x^2 q(x,y) - \varepsilon_1 \partial_y^2 q(x,y) &= (c_2 - c_1) q(x,y) \nonumber\\&\quad+  g[k_{21}](y) p(x-y) \label{eq:q_orig} \\
	\varepsilon_2 \partial_x^2 r(x,y) - \varepsilon_1 \partial_y^2 r(x,y) &= (c_2 - c_1) r(x,y)
\end{align}
subject to the following boundary conditions:
\begin{align}
	\partial_y q(x,x) &= \partial_y r(x,x) + \frac{g[k_{21}](x)}{\varepsilon_2 - \varepsilon_1} \label{eq:pq_if_d}\\
	r(x,x) &= q(x,x) \label{eq:pq_if} \\
	\partial_y q(x,0) &= a^2 p'(x) = a \partial_x q(x,0) \label{eq:pq_bc} \\
	r(x,1) &= 0
\end{align}
In addition, two initial conditions on $r$ can be found from enforcing \eqref{eq:targ_int_2} on \eqref{eq:tfm2}:
\begin{align}
	r(0,y) &= 0 \\
	\partial_x r(0,y) &= \mathbf{1}_{\{y = 0\}}(y) \frac{g[k_{21}](0)}{\varepsilon_2 - \varepsilon_1}
\end{align}
where $\mathbf{1}_{\{y = 0\}}(y)$ is the indicator function equal to 1 on the set $\{y = 0\}$ and 0 otherwise. This is quite unusual when compared to standard backstepping techniques, but is necessary to resolve the condition \eqref{eq:pq_if_d},\eqref{eq:pq_if} at $x = y = 0$. Despite having this unusual initial condition, the target system and gain kernels are unaffected as $\partial_x r(0,y)$ only appears underneath an integration operation.
The kernel equations for the inverse kernels $\bar{p},\bar{q},\bar{r}$ are similar to those of $p,q$ respectively:
\begin{align}
	\bar{p}(x) &= a^{-1} \bar{q}(x,0) \\
	\varepsilon_2 \partial_x^2 \bar{q}(x,y) - \varepsilon_1 \partial_y^2 \bar{q}(x,y) &= (c_1 - c_2) \bar{q}(x,y) \nonumber\\&\quad- g[k_{21}](x) \bar{p}(x-y) \\
	\varepsilon_2 \partial_x^2 \bar{r}(x,y) - \varepsilon_1 \partial_y^2 \bar{r}(x,y) &= (c_2 - c_1) \bar{r}(x,y)
\end{align}
with boundary conditions
\begin{align}
	\partial_y \bar{q}(x,x) &= \partial_y \bar{r}(x,x) + \frac{g[k_{21}](x)}{\varepsilon_1 - \varepsilon_2} \\
	\bar{r}(x,x) &= \bar{q}(x,x) \\
	\partial_y \bar{q}(x,0) &= -a^2 \bar{p}'(x) = -a \partial_x \bar{q}(x,0) \label{eq:pqbar_bc}
\end{align}
As in the forward transformation, initial conditions on $\bar{r}$ can be found:
\begin{align}
	r(0,y) &= 0 \\
	\partial_x r(0,y) &= \mathbf{1}_{\{y = 0\}}(y) \frac{g[k_{21}](0)}{\varepsilon_1 - \varepsilon_2}
\end{align}
An interpretation of the $(p,q,r)$ coupled kernel is that of a hyperbolic PDE $(q,r)$ defined on the square $\mathcal{T} \cup \mathcal{T}_\textrm{u}$, subject to non-local coupling and memory phenomena via $p$. Transmission conditions between $q,r$ exist at the interface $y=x$, where operator $g[k_{21}](x)$ acts as a point forcing through the interface. This interpretation will motivate the well-posedness study of $(p,q,r)$. A point of interest to be raised is on the postulated continuity of the solutions. From \eqref{eq:pq_if}, continuity is imposed between $q,r$, however, due to \eqref{eq:pq_if_d}, the partial derivatives will not exhibit the same property. One may expect piecewise differentiability, in which the derivative loses continuity at the interface $y = x$.

\begin{lemma}
	\label{lem:stability_targtfm}
	The trivial solution $\Omega \equiv 0$ of the target system \eqref{eq:targ2_int_1}-\eqref{eq:targ2_int_3} is exponentially stable in the sense of the $L^2$ norm. That is,
	\begin{align}
		\norm{\Omega(\cdot,t)}_{L^2} \leq \Pi \exp \left( -\gamma(t-t_0) \right) \norm{\Omega(\cdot,t_0)}_{L^2} \label{eq:target_bound}
	\end{align}
	where the constants $\Pi, \gamma$ are given by
	\begin{align}
		\Pi &= a^{-\frac{3}{2}} \\
		\gamma &= \min \{ a^3 c_1, c_2 \} + \frac{\varepsilon_2}{4}
	\end{align}
\end{lemma}
We omit the proof for space, but the result follows in a straightforward manner using the following Lyapunov function:
\begin{align}
		V(t) &:= \int_0^1 \Omega(x,t)^T A \Omega(x,t) dx
	\end{align}
	where $A = \text{diag}(a^3,1)$.

With Lemma \ref{lem:stability_targtfm}, we are equipped to establish state feedback result.
\begin{theorem}
The trivial solution of the system \eqref{eq:model_orgadv_1}-\eqref{eq:model_orgadv_3} is exponentially stable in the sense of the $L^2$ norm under the pair of state feedback control laws $\bar{\mathcal{U}}_1,\bar{\mathcal{U}}_2$:
\begin{align}
	\begin{pmatrix} \bar{\mathcal{U}}_1(t) \\ \bar{\mathcal{U}}_2(t) \end{pmatrix} &= \int_{-1}^1 \begin{pmatrix} F_1(y) \\ F_2(y) \end{pmatrix} \bar{u}(y,t) dy \label{eq:state_feedback_controls}
\end{align}
with feedback gains $F_1,F_2$ defined as
\begin{align}
	F_1(y) &= \begin{cases} (1+y_0)^{-1} k_{11} \left(1,\frac{y_0-y}{1+y_0}\right) & y \leq y_0 \\
													(1-y_0)^{-1} k_{12} \left(1,\frac{y-y_0}{1-y_0}\right) & y > y_0 \end{cases}\label{eq:gain_f1}\\
	F_2(y) &= \begin{cases} (1+y_0)^{-1} h_{1}\left(\frac{y_0-y}{1+y_0}\right) & y \leq y_0 \\
													(1-y_0)^{-1} h_{2}\left(\frac{y-y_0}{1-y_0}\right) & y > y_0 \end{cases}\label{eq:gain_f2}\\
	h_1(y) &= k_{21}(1,y) + q(1,y) - \int_y^1 \bigl[p(1-z) k_{21}(z,y) \nonumber\\&\qquad + q(1,z) k_{11}(z,y) \bigr] dz \\
	h_2(y) &= k_{22}(1,y) + p(1-y) - \int_y^1 \bigl[p(1-z) k_{22}(z,y) \nonumber\\&\qquad+ q(1,z)k_{12}(z,y) \bigr] dz
\end{align}
where $k_{ij},p,q$ are $C^2(\mathcal{T})$ solutions to the kernel equations \eqref{eq:K_mat_PDE},\eqref{eq:p_orig},\eqref{eq:q_orig} respectively (with associated boundary conditions). That is, under the controllers $\bar{\mathcal{U}}_1(t),\bar{\mathcal{U}}_2(t)$, there exists a constant $\bar{\Pi}$ such that
	\begin{align}
		\norm{\bar{u}(\cdot,t)}_{L^2} \leq \bar{\Pi} \exp \left( -\gamma(t-t_0) \right) \norm{\bar{u}(\cdot,t_0)}_{L^2} \label{eq:state_feedback_bound}
	\end{align}
	\label{thm:statefdback}
\end{theorem}

\begin{proof}
	The feedback controllers \eqref{eq:state_feedback_controls} are derived via evaluating transforms \eqref{eq:tfm1},\eqref{eq:tfm2} at the boundary $x = 1$. From \eqref{eq:tfm2}:
	\begin{align}
		W(1,t) &= \mathcal{V}(t) = \begin{pmatrix} 0 \\ \nu_2(t) \end{pmatrix} \nonumber\\&= \begin{pmatrix} 0 \\ \int_0^1 \begin{pmatrix} q(1,y) & p(1-y) \end{pmatrix} W(y,t) dy \end{pmatrix}
	\end{align}
	From \eqref{eq:tfm1},\eqref{eq:control_aux}:
	\begin{align}
		\mathcal{U}(t) &= \int_0^1 K(1,y) U(y,t) dy + \mathcal{V}(t)
	\end{align}
	Or componentwise,
	\begin{align}
		\mathcal{U}_1(t) &= \int_0^1 \begin{pmatrix} k_{11}(1,y) & k_{12}(1,y) \end{pmatrix} U(y,t) dy \\
		\mathcal{U}_2(t) &= \int_0^1 \begin{pmatrix} k_{21}(1,y) & k_{22}(1,y) \end{pmatrix} U(y,t) dy \nonumber\\
		&\quad + \int_0^1 \begin{pmatrix} q(1,y) & p(1-y) \end{pmatrix} U(y,t) dy \nonumber\\
		&\quad - \int_0^1 \int_{0}^{y} K(y,z) U(z,t) dz dy
	\end{align}
	By exchanging the order of integrals in the nested integrals, and applying the inverse folding transformations \eqref{eq:fold_tfm}, the controllers \eqref{eq:state_feedback_controls} can be recovered.

	The proof of the bound \eqref{eq:state_feedback_bound} relies on the well-posedness of the kernel PDEs, studied in Section \ref{sec:wellposed}. Specifically, Lemmas \ref{lem:wellposedk1},\ref{lem:wellposedk2},\ref{lem:wellposedqr} state that continuous solutions exist and are unique. By Morrey's inequality (5.6.2 Theorem 4 in \cite{evans1998PDE}), the continuous embedding $C^1(\mathcal{T}) \subseteq L^2(\mathcal{T})$ holds, and therefore $C^1(\mathcal{T})$ functions possess a bounded $L^2(\mathcal{T})$ norm. The boundedness (in $L^2(\mathcal{T})$) of the kernels $K,p,q,r$ (and their inverses $\bar{K},\bar{p},\bar{q},\bar{r}$, via the bounded inverse theorem) are required, but can be shown from their continuity properties on the compact sets $\mathcal{T},\mathcal{T}_{\textrm{u}}$.

	From \eqref{eq:tfm2}, \eqref{eq:tfm2inv}, one can derive the following equivalence:
	\begin{align}
		M_1^{-1} \norm{W(\cdot,t)}_{L^2}^2 \leq \norm{\Omega(\cdot,t)}_{L^2}^2 \leq \bar{M}_1 \norm{W(\cdot,t)}_{L^2}^2 \label{eq:tfm2equiv}
	\end{align}
	where the coefficient $M_1$ depends on kernels $p,q,r$, and is given by
	\begin{align}
		M_1 &= \left(1 - \norm{q}_{L^2} - \norm{p}_{L^2} - \norm{r}_{L^2}\right)^2
	\end{align}
	and $\bar{M}_1$ is analogous with inverse kernels $\bar{p},\bar{q},\bar{r}$. Similarly, from \eqref{eq:tfm1},\eqref{eq:tfm1inv}, the following equivalence can be derived:
	\begin{align}
		\bar{M}_2^{-1} \norm{U(\cdot,t)}_{L^2}^2 \leq \norm{W(\cdot,t)}_{L^2}^2 \leq M_2 \norm{U(\cdot,t)}_{L^2}^2 \label{eq:tfm1equiv}
	\end{align}
	with
	\begin{align}
		M_2 &= (1 - \norm{K}_{L^2})^2
	\end{align}
	and $\bar{M}_2$ analogous with inverse kernel $\bar{K}$. Then, applying \eqref{eq:tfm2equiv},\eqref{eq:tfm1equiv} to the bound \eqref{eq:target_bound} in Lemma \ref{lem:stability_targtfm}, one can arrive at \eqref{eq:state_feedback_bound}, with
	\begin{align}
		\bar{\Pi} = \left(\sqrt{M_1 \bar{M}_1 M_2 \bar{M}_2}\right) \Pi
	\end{align}

\end{proof}

\subsection{Model transformation for estimation via folding}
In the state-estimation problem, we tackle the related problem to the state-feedback problem. Rather than the controllers existing at either opposing boundary, we establish a problem which has two collocated measurements (of state and flux) in the interior of the PDE at some point $\hat{y}_0$. We also note that the sensor location $\hat{y}_0$ need not to be chosen equal to the control folding point $y_0$. The output, denoted $\mathcal{Y}$, is formulated as
\begin{align}
	\mathcal{Y}(t) = \begin{pmatrix} u(\hat{y}_0,t) \\ \partial_y u(\hat{y}_0,t) \end{pmatrix}
\end{align}

Much like the control case, applying a folding transformation about $\hat{y}_0$ will recover a coupled parabolic system. The transformation
\begin{align}
	\hat{x} &= (\hat{y}_0 - y)/(1 + \hat{y}_0) & y \in (-1,\hat{y}_0) \\
	\hat{x} &= (y - \hat{y}_0)/(1 - \hat{y}_0) & y \in (\hat{y}_0,1)
\end{align}
will admit the following folded states:
\begin{align}
	\check{U}(\hat{x},t) &:= \begin{pmatrix} \check{u}_1(\hat{x},t) \\ \check{u}_2(\hat{x},t) \end{pmatrix} = \begin{pmatrix} u(\hat{y}_0 - (1+\hat{y}_0)\hat{x},t) \\ u(\hat{y}_0 + (1-\hat{y}_0)\hat{x},t) \end{pmatrix}
\end{align}
The evolution of $\check{U}(x,t)$ governed by the following dynamics:
\begin{align}
	\partial_t \check{U}(x,t) &= \check{E} \partial_{x}^2 \check{U}(x,t) + \check{\Lambda}(x) \check{U}(x,t) \label{eq:model_foldobs_1} \\
	\check{\alpha} \partial_x \check{U}(0,t) &= -\beta \check{U}(0,t) \label{eq:model_foldobs_2}\\
	\check{U}(1,t) &= \mathcal{U}(t) \label{eq:model_foldobs_3}
\end{align}
The hat notation on $\hat{x}$ has been dropped for simplicity and the spatial domains are defined within the context of the equation in which it arises. The parameter matrices are then as follows:
\begin{align}
	\check{E} &:= \text{diag}(\check{\varepsilon}_1,\check{\varepsilon}_2) \nonumber\\ &:= \text{diag}\left(\frac{\varepsilon}{(1+\hat{y}_0)^2},\frac{\varepsilon}{(1-\hat{y}_0)^2}\right) \\
	\check{\Lambda}(x) &:= \text{diag}(\check{\lambda}_1(x),\check{\lambda}_2(x)) \nonumber\\ &:= \text{diag}(\lambda(\hat{y}_0 - (1+\hat{y}_0)x), \lambda(\hat{y}_0 + (1-\hat{y}_0)x)) \\
	\check{\alpha} &:= \begin{pmatrix} 1 & \check{a} \\ 0 & 0 \end{pmatrix} \\
	\check{a} &:= (1+\hat{y}_0)/(1-\hat{y}_0)
\end{align}
Certainly, if $\hat{y}_0 = y_0$, then the observation and control folded models are identical.

\subsection{Backstepping state estimator design}
Note that the sensor values in the folded coordinates can be expressed in the following manner:
\begin{align}
	\begin{pmatrix} u(\hat{y}_0,t) \\ \partial_y u(\hat{y}_0,t) \end{pmatrix} &= \begin{pmatrix} u_1(0,t) \\ -(1+y_0)^{-1}\partial_x u_1(0,t) \end{pmatrix} \\ &= \begin{pmatrix} u_2(0,t) \\ (1-y_0)^{-1}\partial_x u_2(0,t) \end{pmatrix}
\end{align}

With the two sensor values collocated at a single point, the design of the state estimator can be uncoupled into two near-identical subproblems. Specifically, we choose the following the estimator structure (indexed by $i \in \{ 1,2\}$):
\begin{align}
	\partial_t \hat{u}_i &= \check{\varepsilon}_i \partial_{x}^2 \hat{u}_i(x,t) + \check{\lambda}_i(x) \hat{u}_i(x,t) \nonumber\\&\quad+ \phi_i(x)\left(\partial_x u_i(0,t) -\partial_x \hat{u}_i(0,t) \right) \label{eq:obs1} \\
	\hat{u}_i(0,t) &= u(0,t) \label{eq:obs3} \\
	\hat{u}_i(1,t) &= \mathcal{U}_i(t) \label{eq:obs4}
\end{align}
We define the error systems of the estimators as $\tilde{u}_i(x,t) := u_i(x,t) - \hat{u}_i(x,t)$. They are governed by
\begin{align}
	\partial_t \tilde{u}_i &= \check{\varepsilon}_i \partial_{x}^2 \tilde{u}_i(x,t) + \check{\lambda}_i(x) \tilde{u}_i(x,t) \nonumber\\&\quad- \phi_i(x) \partial_x \tilde{u}_i(0,t) \label{eq:obserr1} \\
	\tilde{u}_i(0,t) &= 0 \label{eq:obserr3} \\
	\tilde{u}_i(1,t) &= 0 \label{eq:obserr4}
\end{align}
We can then design the $\phi_i(x)$ independently to stabilize trivial solutions $\tilde{u}_i(x,t) \equiv 0$ of the error systems $\tilde{u}_i$. We employ the following pair of backstepping transformations
\begin{align}
	\tilde{w}_i(x,t) &= \tilde{u}_i(x,t) - \int_0^x \Phi_i(x,y) \tilde{u}_i(y,t) dy \label{eq:obstfm}
\end{align}
with the following target systems:
\begin{align}
	\partial_t \tilde{w}_i(x,t) &= \check{\varepsilon}_i \partial_{x}^2 \tilde{w}_i(x,t) - \check{c}_i \tilde{w}_i(x,t) \label{eq:obstarg1} \\
	\tilde{w}_i(0,t) &= 0 \label{eq:obstarg3} \\
	\tilde{w}_i(1,t) &= 0 \label{eq:obstarg4}
\end{align}
The inverse transformations are postulated to be
\begin{align}
	\tilde{u}_i(x,t) &= \tilde{w}_i(x,t) - \int_0^x \bar{\Phi}_i(x,y) \tilde{w}_i(y,t) dy \label{eq:obstfm_inv}
\end{align}
where $\bar{\Phi}_i(x,y)$ will satisfy similar kernel equations to $\Phi$.
\begin{lemma}
	For the choice of coefficients $\check{c}_i > 0, i \in \{1,2\}$, the trivial solutions $(\tilde{w}_i,\tilde{v}_i) \equiv 0$ of the observer error target systems \eqref{eq:obstarg1}-\eqref{eq:obstarg4} are exponentially stable in the $L^2 \times H^1$ sense, that is, there exist coefficients $\check{\Pi}_i, \check{\gamma}_i > 0$ such that for $i \in \{1,2\}$,
	\begin{align}
		\norm{\tilde{w}_i(\cdot,t)}_{L^2} &\leq \check{\Pi}_i \exp\left(-\check{\gamma}_i(t-t_0)\right) \norm{\tilde{w}_i(\cdot,t_0)}_{L^2} \label{eq:obstarg_err_bound}
	\end{align}
	\label{lem:obstarg_stab}
\end{lemma}
The proof of Lemma \ref{lem:obstarg_stab} can be found in \cite{SMYSHLYAEV2005613}.

The companion gain kernel PDEs for $\phi_i$ can be found from imposing conditions arising from \eqref{eq:model_foldobs_1}-\eqref{eq:model_foldobs_3}, \eqref{eq:obstarg1}-\eqref{eq:obstarg4}, and the transformation \eqref{eq:obstfm}.
\begin{align}
	\partial_x^2 \Phi_i(x,y) - \partial_y^2 \Phi_i(x,y) &= -\frac{\lambda_i(x) + c_i}{\varepsilon_i} \Phi_i(x,y) \label{eq:obs_kernel_1}\\
  \Phi_i(1,y) &= 0\label{eq:obs_kernel_2}\\
	\Phi_i(x,x) &= \int_x^1 \frac{\lambda_i(y) + c_i}{2\varepsilon_i} dy \label{eq:obs_kernel_3}
\end{align}
In addition, one additional condition is imposed, which defines the observation gain $\phi_i(x)$ in terms of the transformation kernel $\Phi_i(x,y)$.
\begin{align}
	\phi_i(x) &= -\varepsilon_i \Phi_i(x,0)
\end{align}

\begin{lemma}
	\label{lem:obs_kernel_wellposedness}
	The Klein-Gordon PDEs defined by \eqref{eq:obs_kernel_1}-\eqref{eq:obs_kernel_3} admit unique $C^2(\mathcal{T})$ solutions. As a direct result, the gain kernels $\phi_i$ are bounded in the domain $\mathcal{T}$, that is,
	\begin{align}
		\norm{\Phi_i}_{L^\infty} := \max_{(x,y) \in \mathcal{T}} |\Phi_i(x,y)| \leq \bar{\Phi}_i < \infty
	\end{align}
\end{lemma}
The proof of Lemma \eqref{lem:obs_kernel_wellposedness} is given in \cite{smyshlyaev2005spacetime}.
\begin{remark}
	For the special case $\lambda_i(x) = \lambda_i$ is a constant, an explicit solution to \eqref{eq:obs_kernel_1}-\eqref{eq:obs_kernel_3} can be found:
	\begin{align}
		\Phi_i(x,y) &= -\frac{\lambda_i + c_i}{\varepsilon_i} (1-x) \frac{I_1(z)}{z} \\
		z &= \sqrt{\frac{\lambda_i + c_i}{\varepsilon_i}(2 - x - y)}
	\end{align}
	where $I_1(z)$ is the modified Bessel function of the first kind.
\end{remark}

\begin{theorem}
	Consider the original system \eqref{eq:model_org_1}-\eqref{eq:model_org_3} and the auxiliary observer system defined in \eqref{eq:obs1}-\eqref{eq:obs4} with measurements $u(0,t),\partial_y u(0,t)$. Define the state estimate
	\begin{align}
		\hat{u}(y,t) := \begin{cases} \hat{u}_1\left(\frac{\hat{y}_0 - y}{1 + \hat{y}_0}\right) & y \leq \hat{y}_0 \\ \hat{u}_2\left(\frac{y-\hat{y}_0}{1 - \hat{y}_0}\right) & y > \hat{y}_0 \end{cases} \label{eq:obs_estorig}
	\end{align}
	and
	\begin{align}
		\hat{\bar{u}}(y,t) := \exp \left( -\int_{-1}^y \frac{\nu(z)}{2 \varepsilon} dz \right) \hat{u}(y,t)
	\end{align}
	Then $\hat{\bar{u}}(y,t) \rightarrow \bar{u}(y,t)$ exponentially fast in the sense of the $L^2$ norm, i.e. there exist coefficients $\check{\bar{\Pi}},\check{\gamma}_\text{min} > 0$ such that
	\begin{align}
		&\norm{\bar{u}(\cdot,t) - \hat{\bar{u}}(\cdot,t)}_{L^2} \nonumber\\ &\qquad\leq \check{\bar{\Pi}} \exp\left(-\check{\gamma}_\text{min}(t-t_0)\right) \norm{\bar{u}(\cdot,t_0) - \hat{\bar{u}}(\cdot,t_0)}_{L^2} \label{eq:obs_err_bound}
	\end{align}
	\label{thm:observer}
\end{theorem}
We omit the proof for space, but note that it follows directly from Lemmas \ref{lem:obstarg_stab}, \ref{lem:obs_kernel_wellposedness}. By applying successive inverse transformations, the bound \eqref{eq:obs_err_bound} can be recovered.

\subsection{Output-feedback controller}
The output feedback controller proposed is the composition of the state observer with the state feedback. We state the main result below:
\begin{theorem}[Separation principle]
	Consider the original system \eqref{eq:model_org_1}-\eqref{eq:model_org_3} and the auxiliary observer system defined in \eqref{eq:obs1}-\eqref{eq:obs4} with measurements $u(0,t),\partial_y u(0,t)$. With the state estimate \eqref{eq:obs_estorig}, the feedback controller pair $\bar{\mathcal{U}}_{1}(t),\bar{\mathcal{U}}_2(t)$:
	\begin{align}
		\begin{pmatrix} \bar{\mathcal{U}}_1(t) \\ \bar{\mathcal{U}}_2(t) \end{pmatrix} &= \int_{-1}^1 \begin{pmatrix} F_1(y) \\ F_2(y) \end{pmatrix} \hat{\bar{u}}(y,t) dy \label{eq:outputfdback_law}
	\end{align}
	with the gains $F_1,F_2$ defined in \eqref{eq:gain_f1},\eqref{eq:gain_f2} will stabilize $(\bar{u},\hat{\bar{u}}) \equiv 0$ exponentially in the $L^2$ sense -- that is, there exist constants $\bar{\bar{\Pi}},\bar{\bar{\gamma}} > 0$ such that
	\begin{align}
		\norm{(\bar{u},\hat{\bar{u}})(\cdot,t)}_{L^2} \leq \bar{\bar{\Pi}} \exp(\bar\bar{\gamma}(t-t_0)) \norm{(\bar{u},\hat{\bar{u}})(\cdot,t_0)}_{L^2} \label{eq:outputfdback_bound}
	\end{align}
	\label{thm:outputfdback}
\end{theorem}

\begin{proof}
	The output feedback control law \eqref{eq:outputfdback_law} is rewritten in the $(\bar{u},\tilde{\bar{u}})$ coordinates (recalling that $\tilde{\bar{u}} := \bar{u} - \hat{\bar{u}}$):
	\begin{align}
		\begin{pmatrix} \bar{\mathcal{U}}_1(t) \\ \bar{\mathcal{U}}_2(t) \end{pmatrix} &= \int_{-1}^1 \begin{pmatrix} F_1(y) \\ F_2(y) \end{pmatrix} \bar{u}(y,t) dy + \int_{-1}^1 \begin{pmatrix} F_1(y) \\ F_2(y) \end{pmatrix} \tilde{\bar{u}}(y,t) dy  \label{eq:outputfdback_law_rw}
	\end{align}
	Applying the transformations \eqref{eq:tfm1},\eqref{eq:tfm2} will yield the same target system \eqref{eq:targ2_int_1},\eqref{eq:targ2_int_2}, with the modified boundary condition
	\begin{align}
		\Omega(1,t) &= \int_{-1}^1 \begin{pmatrix} F_1(y) \\ F_2(y) \end{pmatrix} \tilde{\bar{u}}(y,t) dy
	\end{align}
	Applying the Cauchy-Schwarz inequality and the bound from Theorem \ref{thm:observer}, we can bound this boundary condition in the following manner:
	\begin{align}
			\norm{\Omega(1,t)}_2 &\leq \check{\bar{\Pi}} \norm{\begin{pmatrix} F_1(y) \\ F_2(y) \end{pmatrix}}_{L^2} \norm{\tilde{\bar{u}}(\cdot,t_0)}_{L^2} \nonumber\\&\qquad \times \exp\left(-\check{\gamma}_\textrm{min}(t-t_0)\right)  \label{eq:obs_err_bound_input}
	\end{align}
	Following the proof of Lemma \ref{lem:stability_targtfm} and Theorem \ref{thm:statefdback}, one can arrive at the following inequality on the $L^2$ norm of the system state:
	\begin{align}
		\norm{\bar{u}(\cdot,t)}_{L^2} &\leq \bar{\Pi} \exp \left( -\bar{\gamma}(t-t_0) \right) \norm{\bar{u}(\cdot,t_0)}_{L^2} \nonumber\\ &\quad+ \hat{\bar{\Pi}} \exp \bigg(-\min\{\bar{\gamma},\check{\gamma}\}(t-t_0) \bigg)\norm{\tilde{\bar{u}}(\cdot,t_0)}_{L^2} \label{eq:outputfdback_statebound}
	\end{align}
	where
	\begin{align}
		\hat{\bar{\Pi}} &= (1+\norm{p}_{L^2}+\norm{q}_{L^2})(1+\norm{K}_{L^2}) \nonumber\\
		&\qquad \times \norm{\begin{pmatrix} F_1 \\ F_2 \end{pmatrix}}_{L^2} \frac{\check{\bar{\Pi}}}{\sqrt{2|\bar{\gamma}-\check{\gamma}_{\textrm{max}}|}}
	\end{align}
	Taking the root sum square of \eqref{eq:obs_err_bound} and \eqref{eq:outputfdback_statebound}, one can arrive at an exponential stability result for $(\bar{u},\tilde{\bar{u}})$:
	\begin{align}
		\norm{(\bar{u},\tilde{\bar{u}})(\cdot,t)}_{L^2} \leq \tilde{\tilde{\Pi}} \exp \bigg(-\bar{\bar{\gamma}}(t-t_0) \bigg) \norm{(\bar{u},\tilde{\bar{u}})(\cdot,t_0)}_{L^2}
	\end{align}
	where
	\begin{align}
		\tilde{\tilde{\Pi}} &= \max\{\bar{\Pi}, \hat{\bar{\Pi}} + \check{\bar{\Pi}}\} \\
		\bar{\bar{\gamma}} &= \min\{\bar{\gamma},\check{\gamma}\}
	\end{align}
	Finally, transforming back into the $(\bar{u},\hat{\bar{u}})$ coordinates, \eqref{eq:outputfdback_bound} can be recovered, with $\bar{\bar{\Pi}} = 4 \tilde{\tilde{\Pi}}$.
\end{proof}

\section{Gain kernel well-posedness studies}
\label{sec:wellposed}
A necessary and sufficient condition for the invertibility of \eqref{eq:tfm1}, \eqref{eq:tfm2} (and their respective inverse transforms) is the existence of bounded kernels $K,p,q$ on their respective domains. It is not trivially obvious that the kernel PDEs \eqref{eq:K_mat_PDE}-\eqref{eq:K_fold}, \eqref{eq:p_orig},\eqref{eq:q_orig} are well-posed. The goal of this section is to establish and characterize the existence and uniqueness (and regularity) properties of these kernel PDEs.

\subsection{Well-posedness of $K$}
For $K$, we note that the kernel PDE is very similar to that of \cite{vazquez2016coupledpara}, and thus apply an adjusted approach to \eqref{eq:K_mat_PDE}-\eqref{eq:K_xtra}, \eqref{eq:K_fold_compact}. We use the following definition:
\begin{align}
	\check{K}(x,y) &= \sqrt{E} \partial_x K(x,y) + \partial_y K(x,y) \sqrt{E}
\end{align}
which allows us to transform the $2 \times 2$ system of 2nd-order hyperbolic PDE $K$ into the following $2 \times 2 \times 2$ 1st-order hyperbolic PDE system $(K,\check{K})$. Due to $G$ possessing triangular structure \eqref{eq:G_op_def} (a result of Assumption \ref{assum:order}), we can separate the kernel PDEs into cascading sets of PDE systems.
\subsubsection{Well-posedness of first row $K,\check{K}$: ($k_{1i},\check{k}_{1i}$)}
The first set of kernel PDEs we study is $(k_{11},k_{12},\check{k}_{11},\check{k}_{12})$. These kernels comprise an autonomous system of first-order hyperbolic PDEs on a bounded triangular domain, and are \emph{linear and $x$-invariant} PDEs. Thus, our expectation is that the energy of a (potentially weak) solution can only grow (in $x$) at an exponential rate at best.

The component-wise kernels are
\begin{align}
	\sqrt{\varepsilon_1} \partial_x k_{11}(x,y) + \sqrt{\varepsilon_1} \partial_y k_{11}(x,y) &= \check{k}_{11}(x,y) \label{eq:k_11_kernel}\\
	\sqrt{\varepsilon_1} \partial_x k_{12}(x,y) + \sqrt{\varepsilon_2} \partial_y k_{12}(x,y) &= \check{k}_{12}(x,y) \\
	\sqrt{\varepsilon_1} \partial_x \check{k}_{11}(x,y) - \sqrt{\varepsilon_1} \partial_y \check{k}_{11}(x,y) &= (\lambda_1(y) + c_1) k_{11}(x,y) \\
	\sqrt{\varepsilon_1} \partial_x \check{k}_{12}(x,y) - \sqrt{\varepsilon_2} \partial_y \check{k}_{12}(x,y) &= (\lambda_2(y) + c_1) k_{12}(x,y) \label{eq:kc_12_kernel}
\end{align}
with boundary conditions
\begin{align}
	k_{11}(x,0) &= \frac{a\varepsilon_2}{\varepsilon_1(a \varepsilon_2 + \sqrt{\varepsilon_1 \varepsilon_2})} \nonumber\\&\quad\times \int_0^x \sqrt{\varepsilon_1} \check{k}_{11}(y,0) + \sqrt{\varepsilon_2} \check{k}_{12}(y,0) dy \\
	k_{12}(x,0) &= \frac{1}{a \varepsilon_2 + \sqrt{\varepsilon_1 \varepsilon_2}} \nonumber\\&\quad\times \int_0^x \sqrt{\varepsilon_1} \check{k}_{11}(y,0) + \sqrt{\varepsilon_2} \check{k}_{12}(y,0) dy \\
	k_{12}(x,x) &= 0 \\
	\check{k}_{11}(x,x) &= - \frac{\lambda_1(x) + c_1}{2 \sqrt{\varepsilon_1}} \\
	\check{k}_{12}(x,x) &= 0
\end{align}
The system of kernel equations $(k_{1i},\check{k}_{1i})$ is self contained. Due to Assumption \ref{assum:order}, the characteristics of the kernel equations $k_{12},\check{k}_{12}$ will have sub-unity slope, in turn neccessitating two boundary condtions on $k_{12}$ at the $y = 0$ and $y = x$ boundaries.

\begin{lemma}
	The system of first-order hyperbolic PDEs \eqref{eq:k_11_kernel}-\eqref{eq:kc_12_kernel} and associated boundary conditions admit a unique set of $k_{11},k_{12} \in C^2(\mathcal{T})$, $\check{k}_{11},\check{k}_{12} \in C^1(\mathcal{T})$ solutions.

	\label{lem:wellposedk1}
\end{lemma}

\begin{proof}
With a direct application of the method of characteristcs to \eqref{eq:k_11_kernel}-\eqref{eq:kc_12_kernel}, Volterra-type integral equations can be recovered:
\begin{align}
	k_{11}(x,y) &= c_1 a^3 \int_0^{x-y} \sqrt{\varepsilon_1} \check{k}_{11}(z,0) + \sqrt{\varepsilon_2} \check{k}_{12}(z,0) dz \nonumber\\&\quad + \int_0^{\sqrt{\varepsilon_1^{-1}}y} \check{k}_{11}(\sqrt{\varepsilon_1}z + x - y, \sqrt{\varepsilon_1}z) dz \label{eq:k11_soln} \\
	k_{12}(x,y) &= \begin{cases} k_{12,l} & \sqrt{\varepsilon_1} y \leq \sqrt{\varepsilon_2} x \\ k_{12,u} & \sqrt{\varepsilon_1} y \geq \sqrt{\varepsilon_2} x \end{cases} \label{eq:k12_soln} \\
	\check{k}_{11}(x,y) &= - \frac{\lambda_1 \left( \frac{x+y}{2}\right) + c_1}{2 \sqrt{\varepsilon_1}} \nonumber\\ &\quad + \int_0^{\frac{x-y}{2\sqrt{\varepsilon_1}}} \left(\lambda_1\left(-\sqrt{\varepsilon_1}z + \frac{x+y}{2}\right) + c_1 \right) \nonumber\\&\qquad \times k_{11}\left(\sqrt{\varepsilon_1}z + \frac{x+y}{2}, -\sqrt{\varepsilon_1}z + \frac{x+y}{2} \right) dz \label{eq:kc11_soln} \\
	\check{k}_{12}(x,y) &= \int_0^{\frac{x-y}{\sqrt{\varepsilon_1}+\sqrt{\varepsilon_2}}} \lambda_2(-\sqrt{\varepsilon_2}z + \sigma_3(x,y) + c_1) \nonumber\\&\qquad \times k_{12}(\sqrt{\varepsilon_1}z + \sigma_3(x,y), -\sqrt{\varepsilon_2}z + \sigma_3(x,y)) dz \label{eq:kc12_soln}
\end{align}
where $k_{12,u},k_{12,l}$ is defined by
\begin{align}
	k_{12,l}(x,y) &= \int_0^{\sigma_1(x,y)} \sqrt{\varepsilon_1} \check{k}_{11}(z,0) + \sqrt{\varepsilon_2} \check{k}_{12}(z,0) dz \nonumber\\&\quad+ \int_0^{\sqrt{\varepsilon_2^{-1}}y} \check{k}_{12}(\sqrt{\varepsilon_1}z + \sigma_1(x,y),\sqrt{\varepsilon_2}z) dz \\
	k_{12,u}(x,y) &= \int_0^{\frac{x-y}{\sqrt{\varepsilon_1}-\sqrt{\varepsilon_2}}} \check{k}_{12}(\sqrt{\varepsilon_1}z + \sigma_2(x,y),\nonumber\\&\qquad\qquad\qquad\qquad \sqrt{\varepsilon_2}z + \sigma_2(x,y)) dz
\end{align}
and the functions $\sigma_i$ given by
\begin{align}
	\sigma_1(x,y) &= \sqrt{\varepsilon_2^{-1}} (\sqrt{\varepsilon_2} x - \sqrt{\varepsilon_1}y) \\
	\sigma_2(x,y) &= (\sqrt{\varepsilon_1}-\sqrt{\varepsilon_2})^{-1} (\sqrt{\varepsilon_1}y - \sqrt{\varepsilon_2}x) \\
	\sigma_3(x,y) &= (\sqrt{\varepsilon_1}+\sqrt{\varepsilon_2})^{-1} (\sqrt{\varepsilon_2}x + \sqrt{\varepsilon_1}y)
\end{align}
From substituting \eqref{eq:kc12_soln} into \eqref{eq:k12_soln} on the domain $\mathcal{T}_{u} := \{(x,y) \in \R^2 | 0 \leq \sqrt{\varepsilon_2/\varepsilon_1} x \leq y \leq x \leq 1 \}$, one can immediately notice $k_{12,u}(x,y) = \check{k}_{12}(x,y) \equiv 0, (x,y) \in \mathcal{T}_u$.

Using \eqref{eq:k11_soln}-\eqref{eq:kc12_soln}, the following integral equation relations can be established:
\begin{align}
	\begin{pmatrix} k_{11} \\ k_{12} \end{pmatrix} &= \Gamma_1\left[ \begin{pmatrix} \check{k}_{11} \\ \check{k}_{12} \end{pmatrix} \right] := I_1\left[ \begin{pmatrix} \check{k}_{11} \\ \check{k}_{12} \end{pmatrix} \right](x,y) + \Psi_1(x,y) \label{eq:k1_iter}\\
	\begin{pmatrix} \check{k}_{11} \\ \check{k}_{12} \end{pmatrix} &= \Gamma_2\left[ \begin{pmatrix} k_{11} \\ k_{12} \end{pmatrix} \right] := I_2\left[ \begin{pmatrix} k_{11} \\ k_{12} \end{pmatrix} \right](x,y) + \Psi_2(x,y) \label{eq:kc1_iter}
\end{align}
where the operators $\Gamma_1,\Gamma_2$ over $(x,y) \in \mathcal{T})$ encapsulate the affine integral equations \eqref{eq:k11_soln}-\eqref{eq:kc12_soln}, and $I_1,I_2$ represent the linear part in $k,\check{k}$, while $\Psi_1,\Psi_2$ represent the constant part. We establish the following iteration via the method of successive approximations to recover a solution:
\begin{align}
	\begin{pmatrix} k_{11,n+1} \\ k_{12,n+1} \end{pmatrix} &= (\Gamma_1 \circ \Gamma_2)\left[ \begin{pmatrix} k_{11,n} \\ k_{12,n} \end{pmatrix} \right] \label{eq:k1_iter_self}
\end{align}
The existence of a solution $(k_{11},k_{12})$ through the iteration \eqref{eq:k1_iter_self} will imply the existence of a solution $(\check{k}_{11},\check{k}_{12})$ via \eqref{eq:kc1_iter}. To show that this iteration converges, we first define
\begin{align}
	\Delta k_{1,n} := \begin{pmatrix} \Delta k_{11,n} \\ \Delta k_{12,n} \end{pmatrix} :=  \begin{pmatrix} k_{11,n+1} - k_{11,n}  \\ k_{12,n+1} - k_{12,n} \end{pmatrix}\label{eq:delta_k1_def}
\end{align}
Applying \eqref{eq:delta_k1_def} to \eqref{eq:k1_iter_self} and utilizing the properties of affine operators, one can recover the following iteration for $\Delta k_{1,n}$:
\begin{align}
	\Delta k_{1,n+1} = (I_1 \circ I_2)[\Delta k_{1,n}](x,y) \label{eq:delta_k1_iter}
\end{align}
As $\Gamma_1 \circ \Gamma_2$ is a continuous mapping over the complete convex space of bounded continuous functions, then the following statement holds via the Schauder fixed point theorem.
\begin{align}
	\lim_{n \rightarrow \infty} \begin{pmatrix} k_{11,n} \\ k_{12,n} \end{pmatrix} &= \begin{pmatrix} k_{11,0} \\ k_{12,0} \end{pmatrix} + \sum_{n=0}^{\infty} \Delta k_{1,n} = \begin{pmatrix} k_{11} \\ k_{12} \end{pmatrix} \label{eq:k1_limit}
\end{align}
Choosing $k_{11,0} = k_{12,0} = 0$, one can compute the following bound on $\Delta k_{1,0}$ directly:
\begin{align}
	\norm{\Delta k_{1,0}}_1 \leq (\bar{\lambda} + c_1) \left(c_1 a^3 + \varepsilon_1^{-1} + 1\right)x \label{eq:k1_init}
\end{align}
where we have taken the liberty of defining
\begin{align}
	\bar{\lambda} := \max \{ \norm{\lambda_1}_{L^\infty}, \norm{\lambda_2}_{L^\infty} \} = \norm{\lambda}_{L^\infty}
\end{align}
It is important to note that the norm $\norm{\Delta k_{1,0}}_1$ is the \emph{vector} 1-norm and not the $L^1$ function norm. That is,
\begin{align}
	\norm{\Delta k_{1,n}}_1 := |\Delta k_{11,n}(x,y)| + |\Delta k_{12,n}(x,y)|
\end{align}
By using \eqref{eq:k1_init} in \eqref{eq:delta_k1_iter}, one can find the following bound on $\norm{\Delta k_{1,n}}_1$ indexed by iteration $n$:
\begin{align}
	\norm{\Delta k_{1,n}}_1 \leq \frac{2^n ((\bar{\lambda} + c_1)(c_1 a^3 + \varepsilon_1^{-1} + 1))^{n+1}}{(2n+1)!} x^{2n+1}
\end{align}
Due to the bounded domain $\mathcal{T}$, one can find \emph{uniform} convergence properties (where the uniform bound is simply evaluated for $x^{2n+1} \leq 1, \forall x \in [0,1]$).
From \eqref{eq:k1_limit},
\begin{align}
	\norm{\begin{pmatrix} k_{11} \\ k_{12} \end{pmatrix}}_{1} \leq \sum_{n=0}^{\infty} \frac{2^n ((\bar{\lambda} + c_1)(c_1 a^3 + \varepsilon_1^{-1} + 1))^{n+1}}{(2n+1)!} x^{2n+1} \label{eq:k1_gain_bound}
\end{align}
To recover the $C^2(\mathcal{T})$ regularity, one can directly reference \eqref{eq:k1_limit}. Noting that the set $C^n(\mathcal{T})$ is closed under addition for $n \in \N$ along with the \emph{integral} iteration \eqref{eq:delta_k1_iter}, it is easy to see that $\lambda_{1},\lambda_{2} \in C^1([0,1])$ generates the regularity $k_{11}, k_{12} \in C^2(\mathcal{T})$.
\end{proof}

\subsubsection{Well-posedness of second row $K,\check{K}$: ($k_{2i},\check{k}_{2i}$)}
The second set of kernels is $(k_{21},k_{22},\check{k}_{21},\check{k}_{22})$. These feature the kernels $k_{11}, k_{12}$ acting as source terms, however, by employing estimates of $k_{11},k_{12}$ from Lemma \ref{lem:wellposedk1}, we can simplify the system significantly. However, the structure of the problem is different, most notably in how the characteristics evolve.

To account for the different nature of these characteristics, we perform one more transformation on the kernels for $k_{2i}$:
\begin{align}
	\hat{k}_{2i}(x,y) &= \sqrt{\varepsilon_2} \partial_x k_{2i}(x,y) - \sqrt{\varepsilon_i} \partial_y k_{2i}(x,y) \label{eq:kh2_tfm}
\end{align}
where $i \in \{1,2\}$. We then turn our attention to the gain kernel system $(\hat{k}_{21},\check{k}_{21},\hat{k}_{22},\check{k}_{22})$.

The component system of kernel PDEs for $(\hat{k}_{21},\check{k}_{21},\hat{k}_{22},\check{k}_{22})$ is
\begin{align}
	\sqrt{\varepsilon_2} \partial_x \hat{k}_{21}(x,y) + \sqrt{\varepsilon_1} \partial_y \hat{k}_{21}(x,y) &= (\lambda_1(y) + c_2) k_{21}(x,y) \nonumber\\&\quad- g[k_{21}](x) k_{11}(x,y) \label{eq:k_21_kernel}\\
	\sqrt{\varepsilon_2} \partial_x \hat{k}_{22}(x,y) + \sqrt{\varepsilon_2} \partial_y \hat{k}_{22}(x,y) &= (\lambda_2(y) + c_2) k_{22}(x,y) \nonumber\\&\quad- g[k_{21}](x) k_{12}(x,y) \\
	\sqrt{\varepsilon_2} \partial_x \check{k}_{21}(x,y) - \sqrt{\varepsilon_1} \partial_y \check{k}_{21}(x,y) &= (\lambda_1(y) + c_2) k_{21}(x,y) \nonumber\\&\quad- g[k_{21}](x) k_{11}(x,y) \\
	\sqrt{\varepsilon_2} \partial_x \check{k}_{22}(x,y) - \sqrt{\varepsilon_2} \partial_y \check{k}_{22}(x,y) &= (\lambda_2(y) + c_2) k_{22}(x,y) \nonumber\\&\quad- g[k_{21}](x) k_{12}(x,y) \label{eq:kc_22_kernel}
\end{align}
subject to the following boundary conditions:
\begin{align}
	\hat{k}_{21}(x,0) &= -\frac{1-a^2}{1+a^2} \check{k}_{21}(x,0) +  \frac{2a^3}{1+a^2} \check{k}_{22}(x,0) \label{eq:k21_reflection1}\\
	\hat{k}_{22}(x,0) &= \frac{2}{a(1+a^2)} \check{k}_{21}(x,0) + \frac{1-a^2}{1+a^2} \check{k}_{22}(x,0) \\
	\check{k}_{21}(x,x) &= - \frac{\sqrt{\varepsilon_1} - \sqrt{\varepsilon_2}}{\sqrt{\varepsilon_1} + \sqrt{\varepsilon_2}} \hat{k}_{21}(x,x) \label{eq:k21_reflection2}\\
	\check{k}_{22}(x,x) &= - \frac{\lambda_2(x) + c_2}{2 \sqrt{\varepsilon_2}} \label{eq:k21_bc_end}
\end{align}
where the inverse transformations are given to be
\begin{align}
	k_{21}(x,y) &= \frac{1}{2\sqrt{\varepsilon_2}} \int_y^x \check{k}_{21}(z,y) + \hat{k}_{21}(z,y) dz \label{eq:k21_invtfm}\\
	k_{22}(x,y) &= - \int_0^y \frac{\lambda_2(z) + c_2}{2 \sqrt{\varepsilon_2}} dz \nonumber\\&\quad+ \frac{1}{2\sqrt{\varepsilon_2}} \int_y^x \check{k}_{22}(z,y) + \hat{k}_{22}(z,y) dz \label{eq:k22_invtfm}
\end{align}
and the function $g[k_{21}](x)$ can be expressed in terms of $\hat{k}_{21},\check{k}_{21}$:
\begin{align}
	g[k_{21}](x) = \frac{(\varepsilon_2 - \varepsilon_1)}{2 \sqrt{\varepsilon_1}} (\check{k}_{21}(x,x) - \hat{k}_{21}(x,x))
\end{align}
Without the estimates given by Lemma \ref{lem:wellposedk1}, the system of gain kernels would in fact be \emph{nonlinear}, a significantly harder problem.

\begin{lemma}
	The system of first-order hyperbolic PDE \eqref{eq:k_21_kernel}-\eqref{eq:kc_22_kernel} and associated boundary conditions admit a unique set of $\hat{k}_{21},\check{k}_{21},\hat{k}_{22},\check{k}_{22} \in C^1(\mathcal{T})$ solutions.

	\label{lem:wellposedk2}
\end{lemma}

\begin{proof}
	The primary technical difficulty of this proof is incorporating the boundary conditions \eqref{eq:k21_reflection1},\eqref{eq:k21_reflection2}. While in standard integral equation solutions one can apply successive approximations to recover a convergent sum of monomial terms (in increasing powers), the trace term $g[k_{21}](x)$ presents issues with this approach. Thus, we utilize an approach inspired from \cite{auriol2018two},\cite{camacho2017coupled} involving a recursion relating to the finite volume of integration (of the domain $\mathcal{T}$).

	We apply the method of characteristics to \eqref{eq:k_21_kernel}-\eqref{eq:kc_22_kernel} to recover the following system of coupled integro-algebraic equations:

	\begin{align}
		\hat{k}_{21}(x,y) &= \hat{k}_{21}\left( \sigma_4(x,y), 0\right) + \hat{I}_{21}[\hat{k}_{21},\check{k}_{21}](x,y) \label{eq:kh_21_inteq} \\
		\hat{k}_{22}(x,y) &= \hat{k}_{22}\left( x - y, 0\right) + \hat{I}_{22}[\hat{k}_{22},\check{k}_{22},\hat{k}_{21}](x,y) \nonumber\\
		&\quad-\int_0^{\frac{y}{\sqrt{\varepsilon_2}}} \bigg[ \frac{\lambda_2(\sqrt{\varepsilon_2}z)+ c_2}{2\sqrt{\varepsilon_2}} \nonumber\\
		&\qquad\qquad\qquad\times\int_0^{\sqrt{\varepsilon_2} z} (\lambda_2(\xi) + c_2 )d\xi \bigg]dz  \\
		\check{k}_{21}(x,y) &= \check{k}_{21}\left( \sigma_5(x,y), \sigma_5(x,y)\right) + \check{I}_{21}[\hat{k}_{21},\check{k}_{21}](x,y) \\
		\check{k}_{22}(x,y) &= \check{k}_{22}\left( \frac{x+y}{2}, \frac{x+y}{2}\right) + \check{I}_{22}[\hat{k}_{22},\check{k}_{22},\hat{k}_{21}](x,y) \nonumber\\
		&\quad-\int_0^{\frac{x-y}{2\sqrt{\varepsilon_2}}} \bigg[ \frac{\lambda_2\left(-\sqrt{\varepsilon_2} z + \frac{x+y}{2}\right)+ c_2}{2\sqrt{\varepsilon_2}} \nonumber\\
		&\qquad\qquad\qquad\times\int_0^{-\sqrt{\varepsilon_2} z + \frac{x+y}{2}} (\lambda_2(\xi) + c_2 )d\xi \bigg]dz \label{eq:kc_21_inteq}
	\end{align}
	where
	\begin{align}
		\sigma_4(x,y) := x-\frac{\sqrt{\varepsilon_2}}{\sqrt{\varepsilon_1}}y \\
		\sigma_5(x,y) := \frac{\sqrt{\varepsilon_1}x + \sqrt{\varepsilon_2}y}{\sqrt{\varepsilon_1} + \sqrt{\varepsilon_2}}
	\end{align}
	and the integral operators $\hat{I}_{21},\hat{I}_{22},\check{I}_{21},\check{I}_{22}$ are defined
	\begin{align}
		&\hat{I}_{21}[\hat{k}_{21},\check{k}_{21}](x,y) \nonumber\\
		&\quad:= \int_0^{\frac{y}{\sqrt{\varepsilon_1}}} \bigg[ - \frac{\varepsilon_2 - \varepsilon_1}{\sqrt{\varepsilon_1} + \sqrt{\varepsilon_2}} k_{11}(\sqrt{\varepsilon_2} z + \sigma_4(x,y),\sqrt{\varepsilon_1}z) \nonumber\\
		&\qquad\qquad\qquad\qquad \times\hat{k}_{21}(\sqrt{\varepsilon_2} z + \sigma_4(x,y), \sqrt{\varepsilon_2} + \sigma_4(x,y)) \nonumber\\
		&\qquad\qquad+ \frac{\lambda_1(\sqrt{\varepsilon_1}z)+ c_2}{2\sqrt{\varepsilon_2}} \int_{\sqrt{\varepsilon_1}z}^{\sqrt{\varepsilon_2}z + \sigma_4(x,y)} \bigg(\check{k}_{21}(\xi,\sqrt{\varepsilon_1}z) \nonumber\\
		&\qquad\qquad\qquad\qquad\qquad\qquad\qquad+ \hat{k}_{21}(\xi,\sqrt{\varepsilon_1} z)\bigg) d\xi\bigg] dz \\
		&\hat{I}_{22}[\hat{k}_{22},\check{k}_{22},\hat{k}_{21}](x,y) \nonumber\\
		&\quad:= \int_0^{\frac{y}{\sqrt{\varepsilon_2}}} \bigg[ - \frac{\varepsilon_2 - \varepsilon_1}{\sqrt{\varepsilon_1} + \sqrt{\varepsilon_2}} k_{12}(\sqrt{\varepsilon_2} z + x-y,\sqrt{\varepsilon_2}z) \nonumber\\
		&\qquad\qquad\qquad\qquad \times\hat{k}_{21}(\sqrt{\varepsilon_2} z + x-y, \sqrt{\varepsilon_2} + x-y) \nonumber\\
		&\qquad\qquad+ \frac{\lambda_2(\sqrt{\varepsilon_2}z)+ c_2}{2\sqrt{\varepsilon_2}} \bigg( \int_{\sqrt{\varepsilon_2}z}^{\sqrt{\varepsilon_2}z + x-y} \bigg(\check{k}_{22}(\xi,\sqrt{\varepsilon_2}z) \nonumber\\
		&\qquad\qquad\qquad\qquad\qquad\qquad\qquad+ \hat{k}_{22}(\xi,\sqrt{\varepsilon_2} z)\bigg) d\xi \bigg)\bigg] dz \\
		&\check{I}_{21}[\hat{k}_{21},\check{k}_{21}](x,y) \nonumber\\
		&\quad:= \int_0^{\frac{x-y}{\sqrt{\varepsilon_1} + \sqrt{\varepsilon_2}}} \bigg[ - \frac{\varepsilon_2 - \varepsilon_1}{\sqrt{\varepsilon_1} + \sqrt{\varepsilon_2}} k_{11}(\sqrt{\varepsilon_2} z + \sigma_5(x,y)\nonumber\\
		&\qquad\qquad\qquad\qquad\qquad\qquad,-\sqrt{\varepsilon_1}z + \sigma_5(x,y)) \nonumber\\
		&\qquad\qquad\qquad\qquad \times\hat{k}_{21}(\sqrt{\varepsilon_2} z + \sigma_5(x,y), \sqrt{\varepsilon_2} + \sigma_5(x,y)) \nonumber\\
		&\qquad\qquad+ \frac{\lambda_1(-\sqrt{\varepsilon_1}z + \sigma_5(x,y))+ c_2}{2\sqrt{\varepsilon_2}} \nonumber\\
		&\qquad\times\int_{-\sqrt{\varepsilon_1}z + \sigma_5(x,y)}^{\sqrt{\varepsilon_2}z + \sigma_5(x,y)} \bigg(\check{k}_{21}(\xi,-\sqrt{\varepsilon_1}z + \sigma_5(x,y)) \nonumber\\
		&\qquad\qquad\qquad\qquad+ \hat{k}_{21}(\xi,-\sqrt{\varepsilon_1} z+\sigma_5(x,y))\bigg) d\xi\bigg] dz \\
		&\check{I}_{22}[\hat{k}_{22},\check{k}_{22},\hat{k}_{21}](x,y) \nonumber\\
		&\quad:= \int_0^{\frac{x-y}{2\sqrt{\varepsilon_2}}} \bigg[ - \frac{\varepsilon_2 - \varepsilon_1}{\sqrt{\varepsilon_1} + \sqrt{\varepsilon_2}} \nonumber\\
		&\qquad\qquad\qquad \times k_{12}\left(\sqrt{\varepsilon_2} z + \frac{x+y}{2},-\sqrt{\varepsilon_2}z + \frac{x+y}{2}\right) \nonumber\\
		&\qquad\qquad\qquad \times\hat{k}_{21}\left(\sqrt{\varepsilon_2} z + \frac{x+y}{2}, \sqrt{\varepsilon_2} + \frac{x+y}{2}\right) \nonumber\\
		&\qquad\qquad+ \frac{\lambda_2\left(-\sqrt{\varepsilon_2}z + \frac{x+y}{2}\right)+ c_2}{2\sqrt{\varepsilon_2}} \nonumber\\
		&\qquad\qquad\quad\times\bigg(\int_{-\sqrt{\varepsilon_2}z + \frac{x+y}{2}}^{\sqrt{\varepsilon_2}z + \frac{x+y}{2}} \bigg(\check{k}_{22}\left(\xi,-\sqrt{\varepsilon_2}z + \frac{x+y}{2}\right) \nonumber\\
		&\qquad\qquad\qquad\qquad+ \hat{k}_{21}\left(\xi,-\sqrt{\varepsilon_2} z + \frac{x+y}{2}\right)\bigg) d\xi \bigg) \bigg] dz
	\end{align}

	\begin{figure}[tb]
		  \centering
			\includegraphics[width=\linewidth]{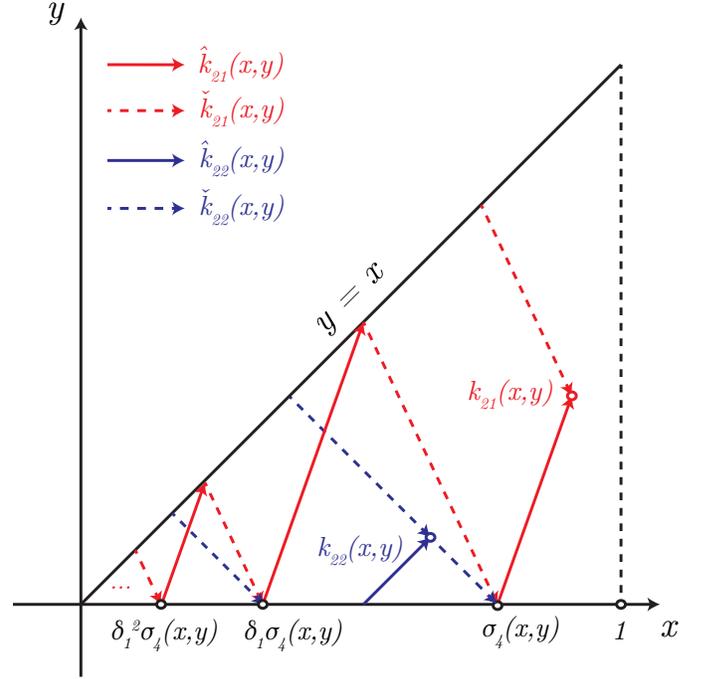}
			\caption{Characteristics of $\hat{k}_{21},\hat{k}_{22},\check{k}_{21},\check{k}_{22}$ featuring an infinite number of reflection boundary conditions.}
			\label{fig:k21}
	\end{figure}

	From enforcing \eqref{eq:k21_reflection1}-\eqref{eq:k21_bc_end} on \eqref{eq:kh_21_inteq}-\eqref{eq:kc_21_inteq} recursively, one can eventually arrive at an integral equation system representation for $(\hat{k}_{21},\hat{k}_{22},\check{k}_{21},\check{k}_{22})$ involving infinite sums of integral operators. The infinite sums appear due to the reflection boundary conditions \eqref{eq:k21_reflection1},\eqref{eq:k21_reflection2} observed in the system.
	\begin{align}
		&\hat{k}_{21}(x,y) = \nonumber\\
		&\quad \lim_{n \rightarrow \infty} \bigg[ - \delta_1^n\delta_2^{n+1} \check{k}_{21} \left( \delta_1^n \sigma_4(x,y),0\right) \nonumber\\
		&\qquad\qquad+ \delta_1^n \delta_2^n \frac{2 a^3}{1 + a^2} \check{k}_{22}(\delta_1^n \sigma_4(x,y),0)\bigg] \nonumber\\
		&\quad- \sum_{n=0}^{\infty} \bigg[ \delta_1^n \delta_2^n \frac{2 a^3}{1 + a^2} \bigg(\frac{\lambda_2(\delta_1^n \sigma_4(x,y))+c_2}{2\sqrt{\varepsilon_2}} \nonumber\\
		&\qquad + \int_0^{\frac{1}{2\sqrt{\varepsilon_2}}\delta_1^n \sigma_4(x,y)} \bigg(\frac{\lambda_2(-\sqrt{\varepsilon_1}z + \frac{1}{2} \delta_1^n \sigma_4(x,y))+c_2}{2\sqrt{\varepsilon_2}} \nonumber\\
		&\qquad\quad\times\int_0^{-\sqrt{\varepsilon_2}z + \frac{1}{2} \delta_1^n \sigma_4(x,y)} (\lambda_2(\xi) + c_2) d\xi\bigg)dz \bigg)\bigg] \nonumber\\
		&\quad+ \sum_{n=0}^{\infty} \bigg[\delta_1^n \delta_2^n \hat{I}_{21}[\hat{k}_{21},\check{k}_{21}](\delta_3^n \sigma_4(x,y), \delta_3^n \sigma_4(x,y)) \nonumber\\
		&\qquad- \delta_1^n \delta_2^{n+1} \check{I}_{21}[\hat{k}_{21},\check{k}_{21}](\delta_1^n \sigma_4(x,y),0) \nonumber\\
		&\qquad+ \delta_1^n \delta_2^n \frac{2 a^3}{1 + a^2} \check{I}_{22}[\hat{k}_{22},\check{k}_{22},\hat{k}_{21}](\delta_1^n \sigma_4(x,y),0) \bigg] \nonumber\\
		&\quad+ \hat{I}_{21}[\hat{k}_{21},\check{k}_{21}](x,y) \label{eq:kh_21_eq}\\
		&\hat{k}_{22}(x,y) = \nonumber\\
		&\quad \frac{2}{a(1+a^2)} \lim_{n \rightarrow \infty} \bigg[ \delta_1^n \delta_2^n \check{k}_{21}(\delta_1^n (x-y),0) \bigg] \nonumber\\
		&\quad- \delta_2 \left( \frac{\lambda_2\left(\frac{x-y}{2}\right)+c_2}{2\sqrt{\varepsilon_2}}\right) \nonumber\\
		&\quad+ \frac{4a^2}{(1+a^2)^2} \sum_{n=1}^{\infty}  \bigg[(-1)^n \delta_1^n \delta_2^n \bigg(\frac{\lambda_2(\delta_1^n(x-y))+c_2}{2\sqrt{\varepsilon_2}} \nonumber\\
		&\qquad-\int_0^{\delta_3^n\frac{x-y}{2\sqrt{\varepsilon_2}}} \bigg[ \frac{\lambda_2\left(-\sqrt{\varepsilon_2} z + \delta_3^n \frac{x-y}{2}\right)+ c_2}{2\sqrt{\varepsilon_2}} \nonumber\\
		&\qquad\qquad\qquad\times\int_0^{-\sqrt{\varepsilon_2} z + \delta_3^n\frac{x-y}{2}} (\lambda_2(\xi) + c_2 )d\xi \bigg]dz\bigg)\bigg] \nonumber\\
		&\quad+ \frac{2}{a(1+a^2)} \sum_{n=1}^{\infty} \bigg[ (-1)^n \delta_1^n \delta_2^{n-1} \nonumber\\
		&\qquad\qquad\qquad\times\hat{I}_{21}[\hat{k}_{21},\check{k}_{21}](\delta_3^n (x-y),\delta_3^n (x-y)) \nonumber\\
		&\qquad + \delta_1^{n-1} \delta_2^{n-1} \check{I}_{21}[\hat{k}_{21},\check{k}_{21}](\delta_1^n (x-y),0) \nonumber\\
		&\qquad + (-1)^n \delta_1^n \delta_2^{n-1} \frac{2a}{1+a^3} \check{I}_{22}[\hat{k}_{22},\check{k}_{22},\hat{k}_{21}](\delta_3^n (x-y),0)\bigg] \nonumber\\
		&\quad+ \delta_2 \check{I}_{22}[\hat{k}_{22},\check{k}_{22},\hat{k}_{21}](x-y,0) + \hat{I}_{22}[\hat{k}_{22},\check{k}_{22},\hat{k}_{21}](x,y) \nonumber\\
		&\quad-\int_0^{\frac{y}{\sqrt{\varepsilon_2}}} \bigg[ \frac{\lambda_2(\sqrt{\varepsilon_2}z)+ c_2}{2\sqrt{\varepsilon_2}}\int_0^{\sqrt{\varepsilon_2} z} (\lambda_2(\xi) + c_2 )d\xi \bigg]dz \nonumber\\
		&\quad-\int_0^{\frac{x-y}{2\sqrt{\varepsilon_2}}} \delta_2 \bigg[ \frac{\lambda_2\left(-\sqrt{\varepsilon_2} z + \frac{x-y}{2}\right)+ c_2}{2\sqrt{\varepsilon_2}} \nonumber\\
		&\qquad\qquad\qquad\times\int_0^{-\sqrt{\varepsilon_2} z + \frac{x-y}{2}} (\lambda_2(\xi) + c_2 )d\xi \bigg]dz \label{eq:kh_22_eq}\\
		&\check{k}_{21}(x,y) = \nonumber\\
		&\quad \lim_{n \rightarrow \infty} \bigg[\delta_1^n \delta_2^n \check{k}_{21}(\delta_1^n \sigma_5(x,y), \delta_1^n \sigma_5(x,y)) \bigg] \nonumber\\
		&\quad + \sum_{n=0}^{\infty} \bigg[ \delta_1^{n+1} \delta_2^n \bigg( \frac{\lambda_2\left(\frac{1}{2}\frac{\delta_1}{\delta_3} \delta_1^n \sigma_5(x,y) \right) + c_2}{2\sqrt{\varepsilon_2}} \nonumber\\
		&\quad-\int_0^{\frac{\delta_1}{\delta_3}\delta_1^n \frac{\sigma_5(x,y)}{2\sqrt{\varepsilon_2}}} \bigg[ \frac{\lambda_2\left(-\sqrt{\varepsilon_2} z + \frac{\delta_1}{\delta_3}\delta_1^n \frac{\sigma_5(x,y)}{2}\right)+ c_2}{2\sqrt{\varepsilon_2}} \nonumber\\
		&\qquad\qquad\qquad\times\int_0^{-\sqrt{\varepsilon_2} z + \frac{\delta_1}{\delta_3}\delta_1^n \frac{\sigma_5(x,y)}{2}} (\lambda_2(\xi) + c_2 )d\xi \bigg]dz \bigg)\bigg] \nonumber\\
		&\quad + \sum_{n=0}^{\infty} \bigg[ - \delta_1^{n+1} \delta_2^n \hat{I}_{21}[\hat{k}_{21},\check{k}_{21}](\delta_1^n \sigma_5(x,y), \delta_1^n \sigma_5(x,y)) \nonumber\\
		&\qquad - \frac{2a^3}{1+a^2} \delta_1^{n+1} \delta_2^n \check{I}_{22}[\hat{k}_{22},\check{k}_{22},\hat{k}_{21}]\left(\frac{\delta_1}{\delta_3}\delta_1^{n} \sigma_5(x,y),0\right) \nonumber\\
		&\qquad  +\delta_1^{n+1} \delta_2^{n+1} \check{I}_{21}[\check{k}_{21},\hat{k}_{21}]\left(\frac{\delta_1}{\delta_3}\delta_1^{n} \sigma_5(x,y),0\right)\bigg] \nonumber\\
		&\quad + \check{I}_{21}[\hat{k}_{21},\check{k}_{21}](x,y) \label{eq:kc_21_eq}\\
		&\check{k}_{22}(x,y) = \nonumber\\
		&\quad -\frac{\lambda_2(\frac{x+y}{2})+c_2}{2\sqrt{\varepsilon_2}} + \check{I}_{22}[\hat{k}_{22},\check{k}_{22},\hat{k}_{21}](x,y) \nonumber\\
		&\quad-\int_0^{\frac{x-y}{2\sqrt{\varepsilon_2}}} \bigg[ \frac{\lambda_2\left(-\sqrt{\varepsilon_2} z + \frac{x+y}{2}\right)+ c_2}{2\sqrt{\varepsilon_2}} \nonumber\\
		&\qquad\qquad\qquad\times\int_0^{-\sqrt{\varepsilon_2} z + \frac{x+y}{2}} (\lambda_2(\xi) + c_2 )d\xi \bigg]dz \label{eq:kc_22_eq}
	\end{align}
	where $\delta_1,\delta_2$ are defined
	\begin{align}
		\delta_1 &= \frac{\sqrt{\varepsilon_1} - \sqrt{\varepsilon_2}}{\sqrt{\varepsilon_1}+\sqrt{\varepsilon_2}} \\
		\delta_2 &= \frac{1-a^2}{1+a^2} \\
		\delta_3 &= \frac{\sqrt{\varepsilon_1}}{\sqrt{\varepsilon_1} + \sqrt{\varepsilon_2}}
	\end{align}
	Since $a < 1, \varepsilon_1 > \varepsilon_2$ as per Assumption \ref{assum:order}, the coefficients $\delta_{1,2,3} \in (0,1)$. It is unclear initially whether the limit and infinite sum terms are convergent, however, as one may notice from Figure \ref{fig:k21}, the contracting volume of integration and the reflection coefficients (appearing in the $\delta_i$ coefficients) will guarantee convergence.

	Like in the proof of Lemma \ref{lem:wellposedk1}, we will define the integral equations \eqref{eq:kh_21_eq}-\eqref{eq:kc_22_eq} in terms of operators for notational compactness. Let $k_2 := \begin{pmatrix} \hat{k}_{21},\hat{k}_{22},\check{k}_{21},\check{k}_{22} \end{pmatrix}$, and
	\begin{align}
		k_2 &= I_3[k_2](x,y) + \Theta[k_2](x,y) + \Psi_3(x,y) \label{eq:k2_int_eq}
	\end{align}
	where $I_3$ is the operator involving the integral operators $\hat{I}_{2i},\check{I}_{2i}$, $\Theta$ is the operator involving limits, and $\Psi_3$ collects the terms independent of $\hat{k}_{2i},\check{k}_{2i}$. We establish an iteration $k_{2,n}$ as
	\begin{align}
		k_{2,n+1} &= I_3[k_{2,n}](x,y) + \Theta[k_{2,n}](x,y) + \Psi_3(x,y) \label{eq:k2_iter}
	\end{align}
	with the iteration residual $\Delta k_{2,n} := k_{2,n+1} - k_{2,n}$ defining the iteration
	\begin{align}
		\Delta k_{2,n+1} &= I_3[\Delta k_{2,n}](x,y) + \Theta[\Delta k_{2,n}](x,y) \label{eq:k2d_iter}
	\end{align}
	We note that \eqref{eq:k2_int_eq} is a continuous mapping over the complete (convex) metric space of bounded continuous functions (via the Schauder fixed point theorem), and make the following statement:
	\begin{align}
		\lim_{n \rightarrow \infty} k_{2,n} = k_{2,0} + \sum_{n=0}^{\infty} \Delta k_{2,n} = k_2 \label{eq:k2_lim}
	\end{align}
	Supposing that $k_{2,0} = 0$,
	\begin{align}
		\norm{\Delta k_{2,0}}_1 = \norm{\Psi_3}_1 &\leq \bigg( \frac{\bar{\lambda}+c_2}{2 \sqrt{\varepsilon_2}} + \frac{1}{2} \bigg( \frac{\bar{\lambda}+c_2}{2 \sqrt{\varepsilon_2}} \bigg)^2\bigg) \nonumber\\
		&\quad \times \bigg( 1 + \frac{1}{1-\delta_1 \delta_2} + \frac{2a^3}{1+a^2}\frac{1}{1-\delta_1\delta_2}\nonumber\\
		&\qquad\qquad + \frac{4a^2}{(1+a^2)^2}\frac{1}{1-\delta_1\delta_2}\bigg) \nonumber\\
		&\leq \bar{\Psi}_{3,0}
	\end{align}
	From iterating $\norm{\Delta k_{2,0}}_1$ through \eqref{eq:k2d_iter}, one can achieve successive bounds on $\Delta k_{2,n}$:
	\begin{align}
		\norm{\Delta k_{2,n}}_1 &\leq \frac{1}{n!} \bigg[ 3\bigg(\bigg(1 + \frac{2}{1-\delta_1 \delta_2}\bigg) + \frac{2}{a(1+a^2)} \frac{1}{1-\delta_1\delta_2}\bigg) \nonumber\\
		&\qquad\times\bigg(a^{-1}\norm{\begin{pmatrix} k_{11} \\ k_{12} \end{pmatrix}}_{1,L^\infty} + \frac{\bar{\lambda} + c_2}{\varepsilon_2}\bigg)  \bigg]^n \bar{\Psi}_{3,0} x^n
	\end{align}
	We remark that it is easy to see for any polynomial bound $\norm{\Delta k_{2,n}}_1$, $\norm{\Theta[\Delta k_{2,n}](x,y)}_1 \leq 0$ due to continuity (and boundedness). Then, via \eqref{eq:k2_lim} and noting $x \in [0,1]$, we can arrive at the following bound on $k_2$:
	\begin{align}
		\norm{k_2}_{1} \leq \sum_{n=0}^{\infty} \norm{\Delta k_{2,n}}_1 \label{eq:k2_gain_bound}
	\end{align}
	which is the power series representation of an exponential bound. The regularity of the solution $k_2$ is also derived from $\eqref{eq:k2_lim}$, where noting that the initial choice of $k_{2,0} = 0$ admits $\Delta k_{2,0} = \Psi_{3}$, which is $C^1(\mathcal{T})$ as it involves sums of $\lambda_2 \in C^1([0,1])$. Then from \eqref{eq:k21_invtfm},\eqref{eq:k22_invtfm}, a single integration yields $k_{21},k_{22} \in C^2(\mathcal{T})$.

\end{proof}

\subsection{Well-posedness of $p,q,r$}
As aformentioned, the $(p,q,r)$-system of kernel PDEs comprise a fairly interesting structure, the heart of which is a wave equation with an interface, whereby forcing is introduced via the differential transmission condition \eqref{eq:pq_if_d} at the interface. It is quite trivial to see that if one can show a solution exists for $q$, then necessarily, a solution $p$ must exist as well.

To faciliate the study of the kernels, we will apply the Riemann invariant transformation as before found in the $K$ kernel. As $q,r$ share congruent characteristics, the solution method is much more straightforward and involves tracing characteristics through the square $\mathcal{T} \cup \mathcal{T}_{\textrm{u}}$.

In this section, we have used the relation \eqref{eq:p_orig} to reduce the $(p,q,r)$ system to $(q,r)$, albeit at the cost of introducing trace terms into the $q$-PDE.

We begin by apply the following definition to derive the Rienmann invariants:
\begin{align}
	\hat{q}(x,y) &= \sqrt{\varepsilon_2} \partial_x q(x,y) - \sqrt{\varepsilon_1} \partial_y q(x,y) \\
	\check{q}(x,y) &= \sqrt{\varepsilon_2} \partial_x q(x,y) + \sqrt{\varepsilon_1} \partial_y q(x,y)
\end{align}
which admit the following coupled PDEs for $(\hat{q},\check{q})$ defined on $\mathcal{T}$:
\begin{align}
	\sqrt{\varepsilon_2} \partial_x \hat{q}(x,y) + \sqrt{\varepsilon_1} \partial_y \hat{q}(x,y) &= I_q[\hat{q},\check{q}](x,y) \label{eq:qhat} \\
	\sqrt{\varepsilon_2} \partial_x \check{q}(x,y) - \sqrt{\varepsilon_1} \partial_y \check{q}(x,y) &= I_q[\hat{q},\check{q}](x,y) \label{eq:qchk}
\end{align}
where the operator $I_q[\hat{q},\check{q}]$ is a linear integral operator defined as
\begin{align}
	I_q[\hat{q},\check{q}](x,y) &= \frac{c_2 - c_1}{2\sqrt{\varepsilon_2}} \int_0^x \check{q}(z,0) dz \nonumber\\
	&\quad+ \frac{c_2 - c_1}{2\sqrt{\varepsilon_1}} \int_0^y (\check{q}(x,z) - \hat{q}(x,z)) dz \nonumber\\
	&\quad+ \frac{a^{-1} g[k_{21}](y)}{2\sqrt{\varepsilon_2}} \int_0^{x-y} \check{q}(z,0) dz
\end{align}

In a similar manner, we define the Riemann invariants for $r$ on $\mathcal{T}_{\textrm{u}}$:
\begin{align}
	\hat{r}(x,y) &= \sqrt{\varepsilon_2} \partial_x r(x,y) - \sqrt{\varepsilon_1} \partial_y r(x,y) \\
	\check{r}(x,y) &= \sqrt{\varepsilon_2} \partial_x r(x,y) + \sqrt{\varepsilon_1} \partial_y r(x,y)
\end{align}
which admits the coupled PDE:
\begin{align}
	\sqrt{\varepsilon_2} \partial_x \hat{r}(x,y) + \sqrt{\varepsilon_1} \partial_y \hat{r}(x,y) &= I_r[\hat{r},\check{r}](x,y) \label{eq:rhat} \\
	\sqrt{\varepsilon_2} \partial_x \check{r}(x,y) - \sqrt{\varepsilon_1} \partial_y \check{r}(x,y) &= I_r[\hat{r},\check{r}](x,y) \label{eq:rchk}
\end{align}
where $I_r[\hat{r},\check{r}]$ is a linear integral operator defined as
\begin{align}
	I_r[\hat{r},\check{r}](x,y) &= \frac{c_2 - c_1}{2\sqrt{\varepsilon_2}} \int_0^x (\hat{r}(z,y) + \check{r}(z,y)) dz
\end{align}

The PDEs given by \eqref{eq:qhat},\eqref{eq:qchk},\eqref{eq:rhat},\eqref{eq:rchk} are subject to the following boundary conditions, which consist of transmission and reflection boundary conditions:
\begin{align}
	\hat{q}(x,0) &= 0 \\
	\check{q}(x,x) &= \check{r}(x,x) - (\sqrt{\varepsilon_1} + \sqrt{\varepsilon_2})^{-1} g[k_{21}](x) \\
	\hat{r}(0,y) &= \check{r}(0,y) = 0 \\
	\hat{r}(x,x) &= \hat{q}(x,x) - (\sqrt{\varepsilon_1} - \sqrt{\varepsilon_2})^{-1} g[k_{21}](x) \\
	\check{r}(x,1) &= -\hat{r}(x,1)
\end{align}
An additional condition employed implicity in the derivation of $(\hat{q},\check{q},\hat{r},\check{r})$ is the following point condition:
\begin{align}
	q(0,0) &= r(0,0) = 0
\end{align}

\begin{lemma}
	The system of first-order hyperbolic PDE \eqref{eq:qhat},\eqref{eq:qchk},\eqref{eq:rhat},\eqref{eq:rchk} with associated boundary conditions admit a unique set of solutions $(\hat{q},\check{q}) \in C^1(\mathcal{T})$ and $(\hat{r},\check{r}) \in C^1(\mathcal{T}_\textrm{u})$.

	\label{lem:wellposedqr}
\end{lemma}

\begin{proof}
We can directly apply the method of characteristics to \eqref{eq:qhat},\eqref{eq:qchk},\eqref{eq:rhat},\eqref{eq:rchk} to recover the following linear integral equations:
\begin{align}
	\hat{q}(x,y) &= \int_0^{\frac{y}{\sqrt{\varepsilon_1}}} I_q[\hat{q},\check{q}](\sigma_6(x,y) + \sqrt{\varepsilon_2} z,\sqrt{\varepsilon_1}z) dz \label{eq:qhat_ie}\\
	\check{q}(x,y) &= \check{r}(\sigma_7(x,y),\sigma_7(x,y)) \nonumber\\
	&\quad- (\sqrt{\varepsilon_1} + \sqrt{\varepsilon_2})^{-1} g[k_{21}](\sigma_7(x,y)) \nonumber\\
	&\quad+ \int_0^{\frac{x - y}{\sqrt{\varepsilon_1} + \sqrt{\varepsilon_2}}} I_q[\hat{q},\check{q}](\sqrt{\varepsilon_2}z + \sigma_7(x,y), \nonumber\\&\qquad\qquad\qquad\qquad\quad -\sqrt{\varepsilon_1} z + \sigma_7(x,y)) dz \label{eq:qchk_ie}
\end{align}
where $\sigma_6, \sigma_7, \sigma_8$ are defined as
\begin{align}
	\sigma_6(x,y) &= x - \sqrt{\frac{\varepsilon_2}{\varepsilon_1}} y \\
	\sigma_7(x,y) &= \frac{\sqrt{\varepsilon_1} x + \sqrt{\varepsilon_2} y}{\sqrt{\varepsilon_1} + \sqrt{\varepsilon_2}}
\end{align}
while for $\hat{r},\check{r}$, we recover piecewise defined linear integral equations which arises due to the mixing of initial and boundary conditions.
\begin{align}
	\hat{r}(x,y) &=
	 	\begin{cases} \hat{r}_l(x,y) & x \leq \sqrt{\varepsilon_2 / \varepsilon_1} y \\ 0 & x > \sqrt{\varepsilon_2 / \varepsilon_1} y \end{cases} \label{eq:rhat_ie}\\
	\check{r}(x,y) &=
		\begin{cases} \check{r}_u(x,y) & x \leq \sqrt{\varepsilon_2 / \varepsilon_1} \\
			\check{r}_l(x,y) & \sqrt{\varepsilon_2 / \varepsilon_1} < x \leq \sqrt{\varepsilon_2 / \varepsilon_1} \\
			0 & x < \sqrt{\varepsilon_2 / \varepsilon_1} y \label{eq:rchk_ie}
		\end{cases}
\end{align}
where
\begin{align}
	\hat{r}_l(x,y) &= \hat{q}(\sigma_8(x,y),\sigma_8(x,y)) \nonumber\\
	&\quad- (\sqrt{\varepsilon_1} - \sqrt{\varepsilon_2})^{-1} g[k_{21}](\sigma_8(x,y)) \nonumber\\
	&\quad+ \int_0^{\frac{y-x}{\sqrt{\varepsilon_1}-\sqrt{\varepsilon_2}}} I_r[\hat{r},\check{r}](\sqrt{\varepsilon_2}z + \sigma_8(x,y),\nonumber\\&\qquad\qquad\qquad\qquad\quad \sqrt{\varepsilon_1}z + \sigma_8(x,y)) dz \\
	\check{r}_u(x,y) &= -\hat{r}(\sigma_9(x,y),1) \nonumber\\
	&\quad+ \int_0^{\frac{1-y}{\sqrt{\varepsilon_1}}} I_r[\hat{r},\check{r}](\sqrt{\varepsilon_2}z + \sigma_9(x,y), \nonumber\\
	&\qquad\qquad\qquad\qquad-\sqrt{\varepsilon_1}z + 1) dz \\
	\check{r}_l(x,y) &= \int_{0}^{\frac{\sqrt{\varepsilon_1} - \sqrt{\varepsilon_2}}{2 \sqrt{\varepsilon_1}}} I_r[\hat{r},\check{r}](\sqrt{\varepsilon_2} z + \sigma_{10}(x,y),\nonumber\\
	&\qquad\qquad\qquad\qquad-\sqrt{\varepsilon_1}z + \sigma_{10}(x,y)) dz
\end{align}
with $\sigma_8,\sigma_9,\sigma_{10}$ defned as
\begin{align}
	\sigma_8(x,y) &= \frac{\sqrt{\varepsilon_1} x - \sqrt{\varepsilon_2} y}{\sqrt{\varepsilon_1} - \sqrt{\varepsilon_2}} \\
	\sigma_9(x,y) &= x + \sqrt{\frac{\varepsilon_2}{\varepsilon_1}} (y - 1) \\
	\sigma_{10}(x,y) &=  x + \sqrt{\frac{\varepsilon_2}{\varepsilon_1}} y
\end{align}

To study the well-posedness $\hat{q},\check{q},\hat{r},\check{r}$ system, it is helpful to study the characteristics geometrically, which are depicted in Figure~\ref{fig:qr}.
\begin{figure}[tb]
		\centering
		\includegraphics[width=\linewidth]{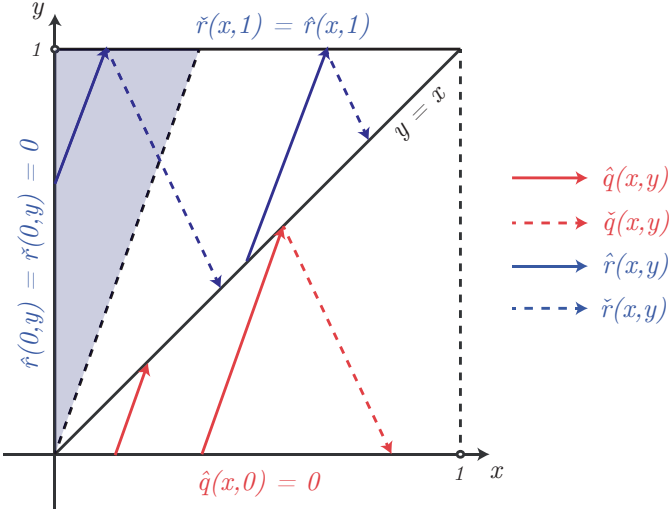}
		\caption{Solution characteristics of $\hat{q},\check{q},\hat{r},\check{r}$. An interface between the solutions exists at $y = x$ defining a jump discontinuity. Because of the initial conditons imposed, $\hat{r} = \check{r} = 0$ in the shaded triangle.}
		\label{fig:qr}
\end{figure}

Much like the analysis of the $K$ kernel, we establish the following operator representation for the affine integral equations that govern $\rho := (\hat{q},\check{q},\hat{r},\check{r})$:
\begin{align}
	\rho &= \Gamma_3[\rho](x,y) := I_4[\rho](x,y) + \Psi_4(x,y)
\end{align}
where analogous to before, $\Gamma_3$ encapsulates the affine integral equations given by \eqref{eq:qhat_ie},\eqref{eq:qchk_ie},\eqref{eq:rhat_ie},\eqref{eq:rchk_ie}. We separate the operator into the linear operator $I_4$ and the constant $\Psi_4$. $\Psi_4$ is evaluated to be
\begin{align}
	\Psi_4(x,y) &:= \begin{pmatrix} 0 &
	\Psi_{4,1}(x,y) &
	\Psi_{4,2}(x,y) &
	\Psi_{4,3}(x,y)
		\end{pmatrix}^T \\
	\Psi_{4,1}(x,y) &= (\sqrt{\varepsilon_1}-\sqrt{\varepsilon_2})^{-1} \nonumber\\
	&\qquad \times g[k_{21}](\sigma_8(\sigma_9(\sigma_7(x,y),\sigma_7(x,y)),1)) \nonumber\\
	&\quad- (\sqrt{\varepsilon_1} + \sqrt{\varepsilon_2})^{-1} g[k_{21}](\sigma_7(x,y)) \nonumber\\
	\Psi_{4,2}(x,y) &= -(\sqrt{\varepsilon_1}-\sqrt{\varepsilon_2})^{-1} g[k_{21}](\sigma_8(x,y)) \nonumber\\
	\Psi_{4,3}(x,y) &= (\sqrt{\varepsilon_1}-\sqrt{\varepsilon_2})^{-1} g[k_{21}](\sigma_8(\sigma_9(x,y),1)) \nonumber
\end{align}
Intuitively, one can understand $\Psi_4$ to represent the nonzero data of the problem. If, perchance the folding point is chosen $y_0 = 0$, then $g[k_{21}] \equiv 0 \Leftrightarrow \Psi_4 \equiv 0$. It is precisely the unmatched artifact from the first transformation, $g[k_{21}]$, that acts as the sole forcing to the $(q,r)$ PDE, as expected.

We carry out the same methodology as for $K$, and establish an iteration $\rho_{k}$ for $n \in \N$:
\begin{align}
	\rho_{n+1} &= I_4[\rho_n](x,y) + \Psi_4(x,y) \label{eq:rho_iter}
\end{align}
The residual $\Delta \rho_{n} := \rho_{n+1} - \rho_n$ will obey the following linear integral equation,
\begin{align}
	\Delta \rho_{n} &= I_4[\Delta \rho_n](x,y) \label{eq:drho_iter}
\end{align}
which arises from abusing the linear property of $I_4$. We note that in the complete space of bounded continuous functions, the iteration \eqref{eq:rho_iter} will converge (via the Schauder fixed point theorem) if we can show uniform Cauchy convergence. The iteration limit thus can be rewritten as an infinite summation:
\begin{align}
	\lim_{n \rightarrow \infty} \rho_n &= \rho_0 + \sum_{n=0}^{\infty} \Delta \rho_n = \rho \label{eq:rho_lim}
\end{align}
It is quite clear from imposing \eqref{eq:rho_iter},\eqref{eq:drho_iter} that by choosing $\rho_0 = 0$, $\Delta \rho_0$ can be computed to be
\begin{align}
	\Delta \rho_0 &= \Psi_4 \nonumber\\
	\Rightarrow \norm{\Delta \rho_0}_1 &= \norm{\Psi_4}_1 \leq \frac{4 \sqrt{\varepsilon_1} + 2 \sqrt{\varepsilon_2}}{\varepsilon_1 - \varepsilon_2} \norm{g[k_{21}]}_{L^\infty} \label{eq:drho_0}
\end{align}
From using \eqref{eq:drho_0} in the iteration \eqref{eq:drho_iter}, one can find the successive bounds on the residuals:
\begin{align}
	\norm{\Delta \rho_{n}}_1 \leq \frac{1}{n!} \left(\frac{2 \varepsilon_2^{-1}(2|c_1 - c_2| + \norm{g[k_{21}]}_{L^{\infty}}}{2} x \right)^{n} \norm{\Psi_4}_1 \label{eq:drho_boundn}
\end{align}
Noting \eqref{eq:rho_lim}, it is quite trivial to see that $\rho$ is bounded (in vector 1-norm) by an exponential. This guarantees the existence of a solution $\rho$, and thus $(\hat{q},\check{q},\hat{r},\check{r})$ admit a solution. In fact, due to the linearity of the PDEs \eqref{eq:qhat},\eqref{eq:qchk},\eqref{eq:rhat},\eqref{eq:rchk}, it is not difficult to show that this solution is also unique.
\end{proof}

\section{Folding point analysis and numerical study}
\label{sec:simulation}

\begin{table}[b]
	\centering
	\begin{tabular}{c|c}\hline
				Parameter             & Value \\ \hline
				$\varepsilon$ & $1$ \\
				$\lambda(x)$ & $-4x^2 - 2 x + 6$ \\
				$y_0$ & $-0.05, -0.30$\\
				$\hat{y}_0$ & $0.05, -0.45$ \\
				$c_1, c_2$ & $5$\\
				$\check{c}_1, \check{c}_2$ & $1$ \\
				\hline
	\end{tabular}
	\vspace{1.2mm}
	\caption{Simulation parameters}
	\label{tab:params}
\end{table}

The parameters chosen for simulation are given in Table \ref{tab:params}. $\lambda$ is specifically chosen to not be symmetric about $x = 0$, and actually attains a maximum at $x = -0.25$. This is motivated by the intuition that choosing a folding point not at the point of symmetry $x = 0$ may afford better performance according a preferred index. It is also important to note that $c_1, c_2, \check{c}_1,\check{c}_2$ all influence the system response in some manner that is not wholly independent from the choices $y_0,\hat{y}_0$.

\subsection{Folding point selection}
\begin{figure}[tb]
	  \centering
		\includegraphics[width=\linewidth]{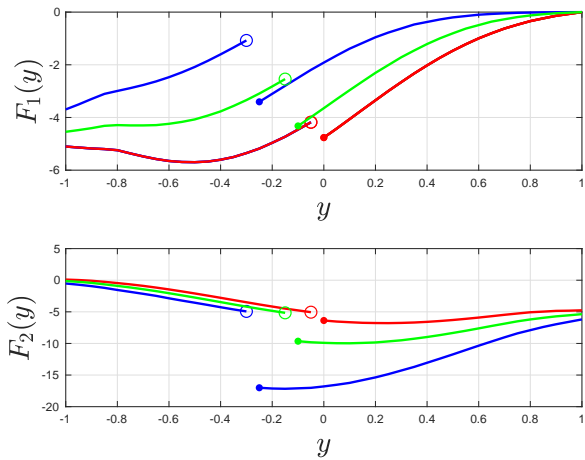}
		\caption{Numerically computed gains for given reaction coefficient $\lambda$ for three separate folding cases: \textbf{(red)} $y_0 = -0.05$, \textbf{(green)} $y_0 = -0.15$, \textbf{(blue)} $y_0 = -0.30$}
		\label{fig:controller_gains}
\end{figure}

It is difficult to directly characterize the size of the controller gains $F_1,F_2$ (defined in \eqref{eq:gain_f1},\eqref{eq:gain_f2}), but one may glean intuition for how the controllers grow based upon what the bounds on the gain kernels suggest.

As one may note from Figure \ref{fig:controller_gains}, the control gains are not necessarily continuous at the selected folding point. Surely, as $y_0 \rightarrow 0$ (the point of symmetry), one recovers the continuous case. However, as the folding point is chosen to be more and more biased (for the same set of target system reaction coefficients $c_i$), one control gain grows smaller (less effort) while the other is magnified (more effort).

Although not provable, the bounds \eqref{eq:k1_gain_bound},\eqref{eq:k2_gain_bound},\eqref{eq:drho_boundn} suggest this behavior as well. In \eqref{eq:k1_gain_bound}, the bound on the controller gains $k_{1i}$ arise as an exponential in $a^3$. In \eqref{eq:k2_gain_bound},\eqref{eq:drho_boundn}, the control gains $k_{2i},p,q$ are parametrized (exponentially as well) in $a^{-1}$.

\subsection{Numerical results for output-feedback}

\begin{figure}[tb]
	  \centering
		\includegraphics[width=\linewidth]{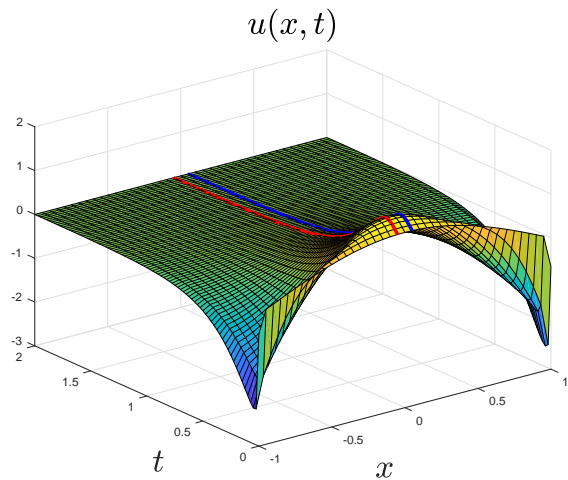}
		\caption{Closed-loop response $u(x,t)$ with folding points chosen to be $y_0 = -0.05, \hat{y}_0 = 0.05$.}
		\label{fig:cl_response_1}
\end{figure}

\begin{figure}[tb]
	  \centering
		\includegraphics[width=\linewidth]{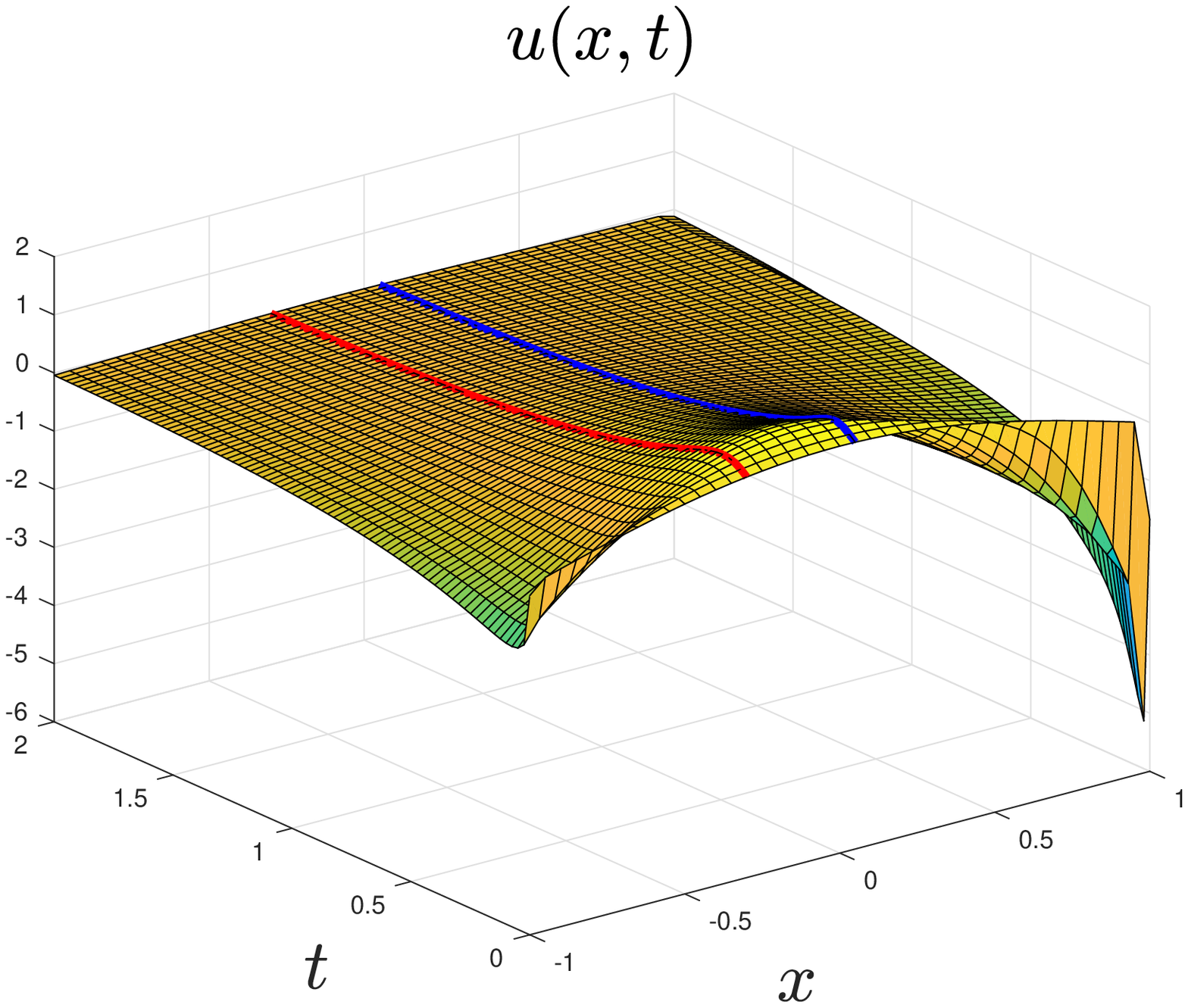}
		\caption{Closed-loop response $u(x,t)$ with folding points chosen to be $y_0 = -0.30, \hat{y}_0 = 0.05$.}
		\label{fig:cl_response_2}
\end{figure}

\begin{figure}[tb]
	  \centering
		\includegraphics[width=\linewidth]{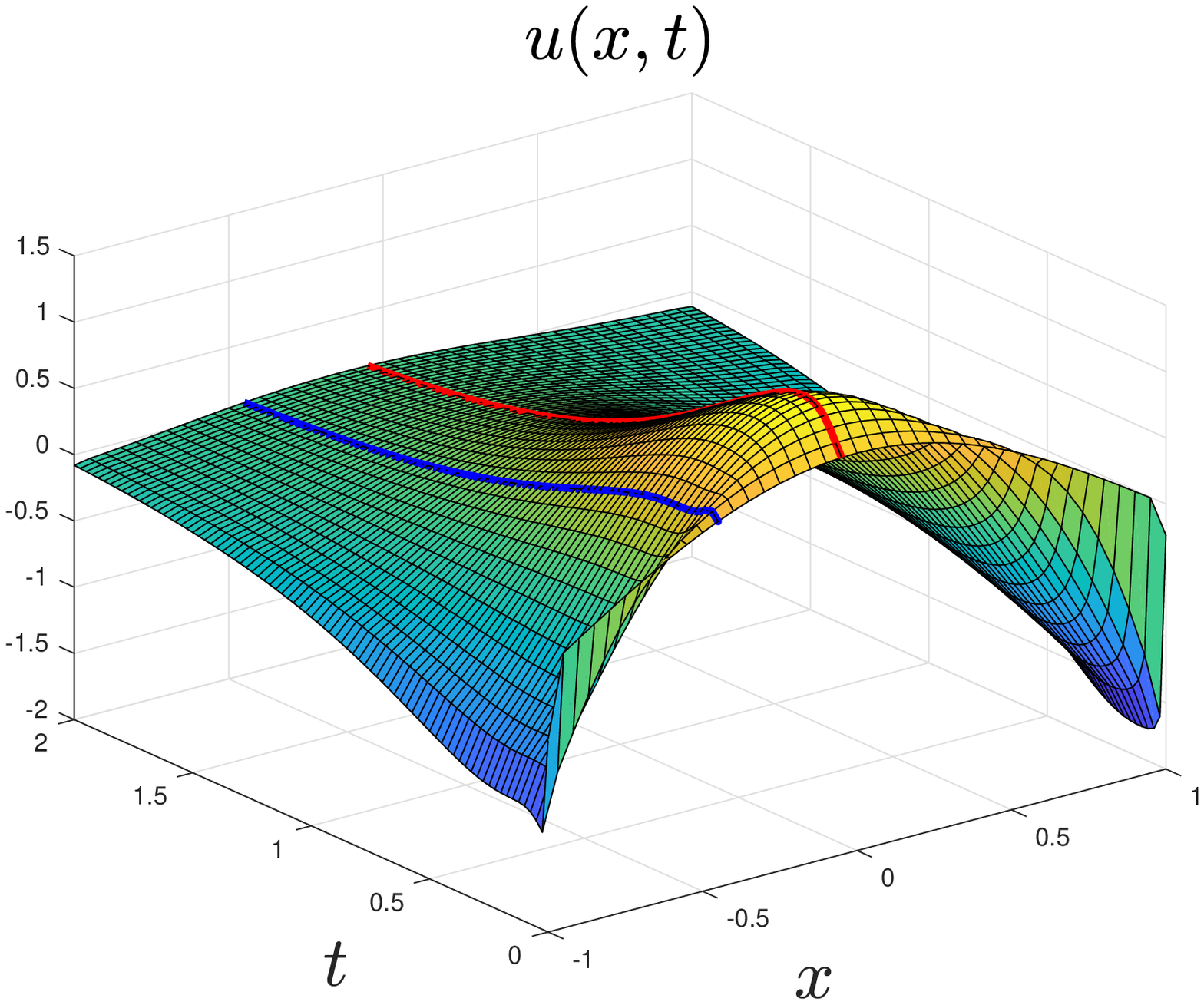}
		\caption{Closed-loop response $u(x,t)$ with folding points chosen to be $y_0 = -0.05, \hat{y}_0 = -0.45$.}
		\label{fig:cl_response_3}
\end{figure}

\begin{figure}[tb]
	  \centering
		\includegraphics[width=\linewidth]{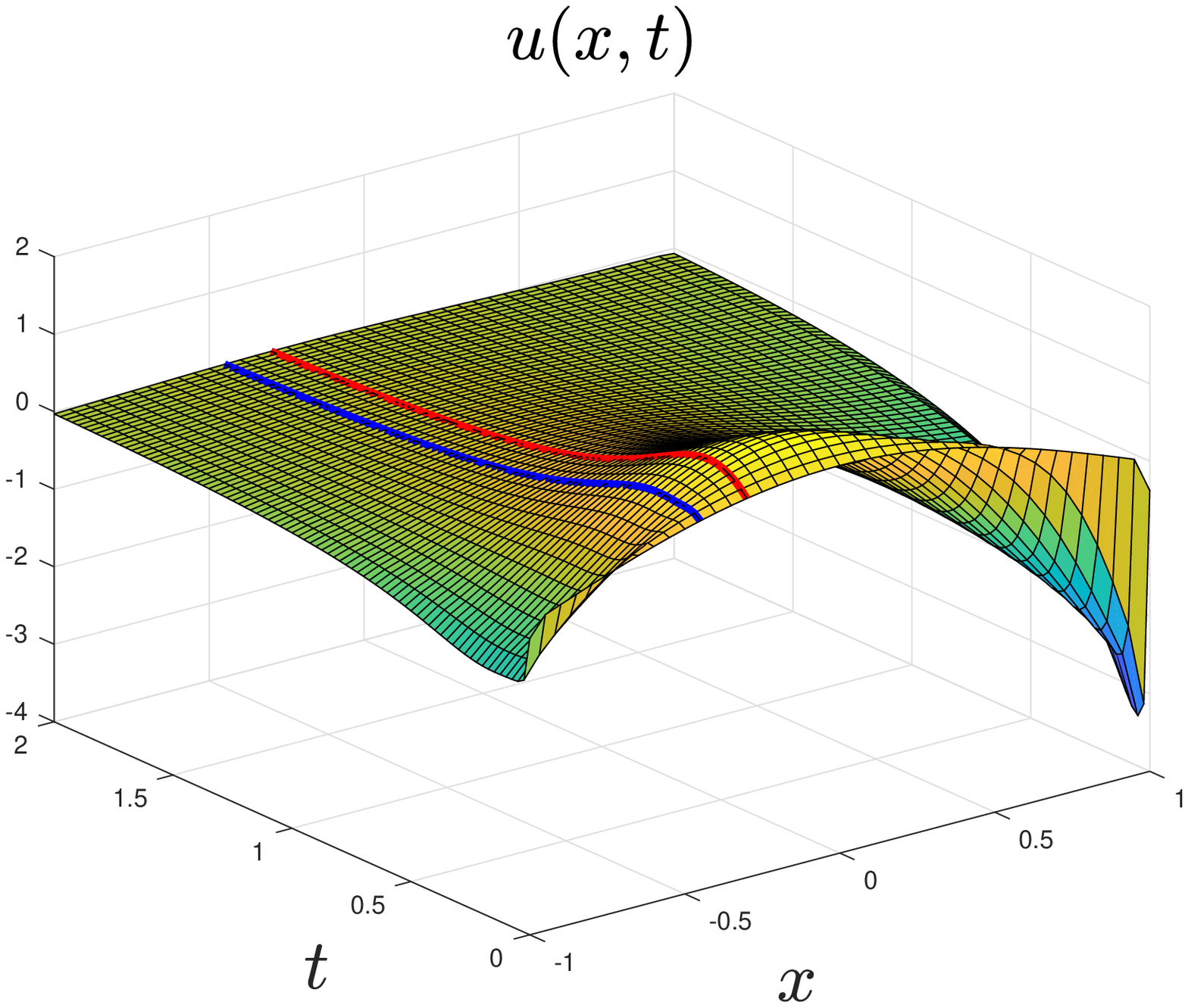}
		\caption{Closed-loop response $u(x,t)$ with folding points chosen to be $y_0 = -0.30, \hat{y}_0 = -0.45$.}
		\label{fig:cl_response_4}
\end{figure}

%
%
%

\begin{figure}[tb]
	  \centering
		\includegraphics[width=\linewidth]{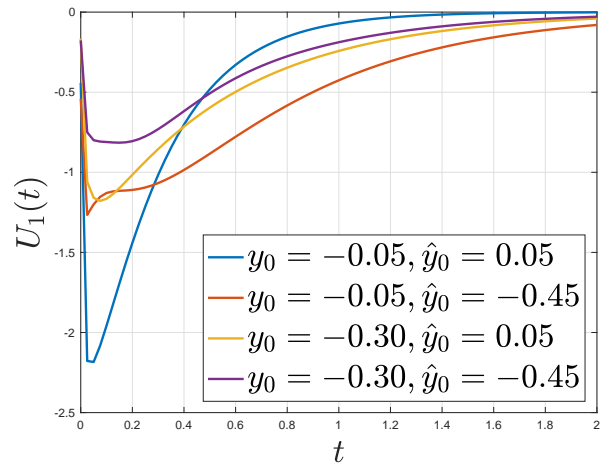}
		\caption{Comparison of control effort by left controller ($\mathcal{U}_1(t)$) over different folding choices}
		\label{fig:control_1}
\end{figure}

\begin{figure}[tb]
	  \centering
		\includegraphics[width=\linewidth]{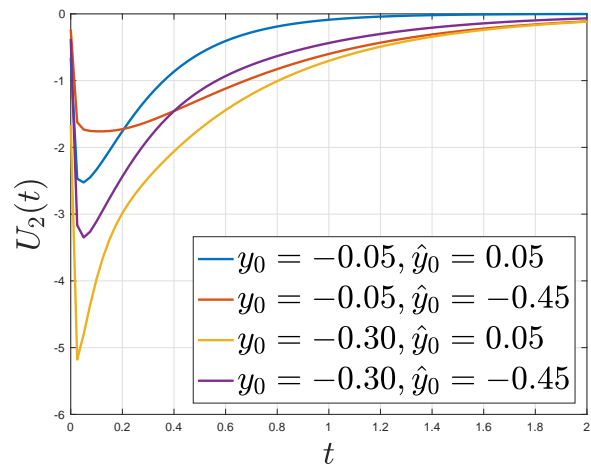}
		\caption{Comparison of control effort by right controller ($\mathcal{U}_2(t)$) over different folding choices}
		\label{fig:control_2}
\end{figure}

Due to the choice of a sufficiently large positive $\lambda$, the open-loop system is unstable and therefore necessitates feedback control. Two choices of control folding points $y_0$ and two choices of measurement points $\hat{y}_0$ are simulated, with the control folding point $y_0$ marked in red and the measurement point $\hat{y}_0$ marked in blue.

Comparing Figures \ref{fig:cl_response_1}, \ref{fig:cl_response_3} with Figurues \ref{fig:cl_response_2},\ref{fig:cl_response_4} gives insight to how changing the control folding point affects the response -- the controller $\mathcal{U}_1$ has a lower peak value in the biased case (Figure \ref{fig:cl_response_2},\ref{fig:cl_response_4}) than that of the close-to-symmetric case (Figure \ref{fig:cl_response_1},\ref{fig:cl_response_3}). However, it is quite clear to note that the controller $\mathcal{U}_2$ pays the cost in having a much higher peak value.

Comparing Figures \ref{fig:cl_response_1},\ref{fig:cl_response_2} with Figures \ref{fig:cl_response_3},\ref{fig:cl_response_4} gives insight to how changing the measurement point $\hat{y}_0$ affects the system response. One can note that the closer the measurement is to the boundary, the performance will improve (uniformly).

The controller responses are given in Figures \ref{fig:control_1},\ref{fig:control_2}. It can be noted that the selection of the control folding point appears to suggest an inherent waterbed effect in $L^2$ versus $L^\infty$ (in time. The numerical simulations suggest that as $y_0 \rightarrow -1$ (the biased case), the controller improves in the $L^2$ sense at the cost of the peak value. Conversely, as $y_0 \rightarrow 0$ (the symmetric case), the controller improves in the $L^\infty$ sense at the cost of the convergence speed (related to $L^2$).

\section{Conclusion}
A methodology for designing output feedback bilateral boundary controllers for linear parabolic class PDEs is generated as the main result of the paper. Compared with existing bilateral control designs for parabolic PDEs, the folding appraoch affords additional design degrees of freedom in not only control but also estimator design.

The primary advantage that the folding approach admits is a generalization of bilateral control design. A design for a given performance index e.g. energy ($L^2$) or boundedness ($L^\infty$) can be achieved in a straightforward manner. The unilateral control design is recoverable in the limit from the folding control design; therefore, the design is far more flexible as a methodology.

Without explicit solutions to the gain kernel equations, the effect of the design parameters on system response is difficult to quantify. However, numerical analysis is given which suggests at least qualitative intuition for selecting folding points for desired behavior. A waterbed effect is noted, in which the controller energy ($L^2$) exhibits an inherent trade off with boundedness ($L^\infty$).

The observer result as a stand alone result is particularly interesting theoretically, as it no longer technically falls within the ``boundary'' observer design any more. Of future interest is developing methodology for design of estimators with arbitrarily placed measure zero measurments in the interior. One may even begin to ask more fundamental questions about conditions about the number of measurements needed to make and allowable locations, because it is not immediately obvious how either affects the observability of the system.

The folding approach also opens the door to potential results involving 1-D PDEs exhibiting coupling structures at points on the interior, as opposed to spatially distributed coupling or boundary coupling. An extension to the folding framework, involving an unstable ODE coupled on an arbitrary point on the interior, is explored in \cite{chen2019foldingode}. Certainly, one may begin to explore additional couplings, which involve feedback coupling between the unstable ODE and the parabolic PDE, or even coupling other 1-D PDEs at the boundary.

\label{sec:conclusions}

\bibliographystyle{IEEETranS}
\bibliography{ref2}

\end{document}